\documentclass[11pt,reqno]{amsart}

 \usepackage{amsthm, amssymb, amsmath, amsfonts, mathtools, mathrsfs}
\usepackage{ dsfont }
\usepackage[all,cmtip]{xy}
\usepackage{comment}
\usepackage{titlesec}
\usepackage{lineno, tipa, color}
\usepackage{pst-node}
\usepackage{enumitem}
\usepackage{hyperref}
\usepackage{tikz-cd} 
\allowdisplaybreaks
\newtheorem{theorem}{Theorem}[section]
\newtheorem{lemma}[theorem]{Lemma}
\newtheorem{corollary}[theorem]{Corollary}
\newtheorem{proposition}[theorem]{Proposition}

\newtheorem{remark}[theorem]{Remark}

\theoremstyle{definition}

\newcommand{\m}{\mathfrak m}
\newcommand{\p}{\mathfrak p}

\newcommand{\q}{\mathfrak q}

\DeclareMathOperator{\height}{height}

\DeclareMathOperator{\Spec}{Spec}
\def\Ext{\rm{Ext}}

\def\U{{\mathcal U}}
\def\R{{\mathcal R}}
\def\F{{\mathcal F}}

\def\M{{\mathcal M}}
\def\I{{\mathcal I}}
\def\L{{\mathcal L}}
\newtheorem{mainthm}{Theorem}

\newtheorem*{remark*}{Remark}

\DeclareMathOperator{\h}{H}

\DeclareMathOperator{\depth}{depth}
\DeclareMathOperator{\ann}{Ann}
\DeclareMathOperator{\im}{Im}
\DeclareMathOperator{\coker}{Coker}
\DeclareMathOperator{\g-reg}{g-reg}
\DeclareMathOperator{\ass}{Ass}

\titleformat{\section}
[block]{\normalfont\bfseries\Large}{\rlap{\thesection.}}{0em}
{\hspace{.05\linewidth}}

\titleformat{\subsection}[block]{\Large\bfseries}{\rlap{\thesubsection.}}{0.7em}{\hspace{.05\linewidth}}
\numberwithin{equation}{section}
\begin{document}
	\noindent
	\title{Multiplicity versus Buchsbaumness of  the special fiber ring }
	
	\author{Anoot Kumar Yadav and Kumari Saloni}
    	\address{Department of Mathematics
			IIT Patna, Bihta  Patna, 801106, India
			}
	\email{anoot\_2021ma06@iitp.ac.in, vicky.anoot@gmail.com}
	\address{Department of Mathematics,
			IIT Patna, Bihta  Patna, 801106, India
			}
	\email{ksaloni@iitp.ac.in}
    
	\begin{abstract}
 In this paper, we study the passage of Buchsbaumness from a Noetherian local ring $(A,\m)$ to the blow-up algebras,  $G(I)=\mathop{\oplus}\limits_{n\geq 0}I^n/I^{n+1}$ and  $F_{\m}(I)=\mathop{\oplus}\limits_{n\geq 0} I^n/\m I^n$, for an $\m$-primary ideal $I$.  Suppose $J\subseteq I$ is a minimal reduction of $I$. In Noetherian local rings, it is true that the fiber multiplicity has an upper bound: $f_0(I)\leq e_1(I)-e_1(J)-e_0(I)+\ell(A/I)+\mu(I)-\dim A+1$ and the first Hilbert coefficient satisfies: $e_1(I)-e_1(J)\geq 2e_0(I)-2\ell(A/I)-\ell(I/(I^2+J))$ where $\ell(*)$ denotes length. Assume the boundary conditions in both the above. Then under mild conditions, we prove that if $A$ is generalized Cohen-Macaulay, then $F_\m(I)$ is generalized Cohen-Macaulay and geometric regularity of $F_\m(I)=\left\lfloor \frac{d}{2}\right\rfloor +1+d$. In addition, if $\depth A>0$, then $\depth F_\m(I)=\depth A$. Further, if $A$ is Buchsbaum and $\depth A\geq d-1$, then  $F_\m(I)$ is Buchsbaum. To handle the special fiber ring through exact sequences, we generalize earlier results of Corso and Ozeki on the Buchsbaumness of the associated graded rings $G(I)=\mathop\oplus_{n\geq 0} I^n/I^{n+1}$ to the case of certain module filtration. In particular, we prove that if $A$ is Buchsbaum (generalized Cohen-Macaulay respectively) and $e_1(\I)-e_1(J)=2e_0(\I)-2\ell(A/I_1)-\ell(I_1/(I_2+J))$ where $\I$ is a multiplicative $I$-good filtration and $J\subseteq I$ is a standard parameter ideal, then under some mild conditions, the associated graded ring $G(\I)$  is Buchsbaum (generalized Cohen-Macaulay respectively) with the preserved $\mathds{I}$-invariant, i.e., $\mathds{I}(G(\I))=\mathds{I}(A)$.

	\end{abstract}
	\maketitle
	
	\section{Introduction} 
Let $(A,\m)$ be a Noetherian local ring of dimension $d>0$, $I$ an $\m$-primary ideal of $A$ and $M$ be a finitely generated faithful $A$-module of dimension
    $d$. Let $\I$ be a multiplicative $I$-good filtration of $A$, i.e, $\I$ is a sequence of ideals $\I=\{I_n\}_{n\geq 0}$ satisfying $I_{n+1}\subseteq I_n$, $I_nI_m\subseteq I_{n+m}$ for all $m, n \geq 0$ and $II_n=I_{n+1}$ for all $n\gg 0 $. The associated graded ring $G(\I)=\mathop{\oplus}\limits_{n\geq 0} I_n/I_{n+1}$, the Rees ring $R(\mathcal{I})=\oplus_{n=0}^\infty I_nt^n$ and the special fiber ring $F_{\m}(I)=\mathop\oplus\limits_{n\geq 0}I^n/\m I^n$ are referred to as the blow-up algebras. When $\I$ is the $I$-adic filtration $\{I^n\}$, we write $G(I)$ and $R(I)$ for the associated graded ring and the Rees ring, respectively. 
	
	The structural properties of the local ring and the blow-up algebras are greatly connected to each other. The local ring $A$ inherits many important properties, such as reducedness, normality, and Cohen-Macaulayness, from the associated graded ring. 
	However, this is not true for Buchsbaum rings. See \cite[Example 4.10]{G84} for an example where $G(I)$ is Buchsbaum and $A$ is not Buchsbaum. A finite $A$-module $M$ is said to be a {\textit{Buchsbaum module}} if the invariant $\ell_A(M/JM)-e_0^J(M)$ is independent of the choice of parameter ideals $J$ of $M$ where $\ell_A(M/JM)$ and $e_0^J(M)$ denote the length of $M/JM$ and the multiplicity of $M$ relative to $J$ respectively. The above condition is equivalent to saying that every system $x_1,\ldots, x_d$ of parameters of $M$ forms a weak $M$-sequence, that is
	$ (x_1,\ldots,x_{i-1})M:_M x_i=(x_1,\ldots,x_{i-1})M:_M \m$ for $1\leq i\leq d.$  A module $M$ is Cohen-Macaulay if and only if it is Buchsbaum with $\ell_A(M/JM)-e_0^J(M)=0$ for every parameter ideal $J$ of $M$. Buchsbaum modules appear in abundance. In fact, it was proved in \cite[Theorem 1.1]{Go1} that for given integers $d>0$ and $h_i\geq 0$ for $0\leq i\leq d-1$,  there exists a Buchsbaum local ring $A$ with $\dim A=d$ and $h_i=\ell_A(\h_\m^i(A))$ for $0\leq i\leq d-1$ where $\h_\m^i(A)$ denote the $i$-th local cohomology module of $A$ with support in $\m$.

    	We are interested in the inheritance of ring-theoretic properties while passing from $A$ to the blow-up algebras. For example, suppose $A$ is a Cohen-Macaulay local ring. If the reduction number of $\I$ is at most one, then $G(\I)$ is Cohen-Macaulay.

        In a Noetherian ring, it is true that
    \begin{equation}\label{very-1st-eqn}
		 e_1(\I)-e_1(J)\geq 2e_0(\I)-2\ell(A/I_1)-\ell(I_1/(I_2+J)) \end{equation}
     where $J=(a_1,\ldots,a_d)\subseteq I$ is a minimal reduction of $\I$, see
	 \cite[Theorem 2.4]{rossi-valla}. 
     When $A$ is Cohen-Macaulay, 
	 authors in \cite{elias-valla} and \cite{GR} proved that the equality  in \eqref{very-1st-eqn}     
    holds true if and only if $I_{n+1}=JI_n$ for all $n\geq 2$ and $J\cap I_2=JI_1$. In this case,  $G(\I)$ is Cohen-Macaulay. 
    Now suppose that $A$ is a Buchsbaum local ring and \begin{equation}\label{very-2nd-eqn}
		e_1(I)-e_1(J)= 2e_0(I)-2\ell(A/I)-\ell(I/(I^2+J)). \end{equation}
	Then Corso \cite{Corso09} proved that $G(I)$  is a Buchsbaum ring in case of $d=1$ and $I=\m$.  Rossi and Valla in \cite[Theorem 2.1]{rossi-valla}
		 generalized Corso's result for modules. In \cite{ozeki12} and  \cite{ozeki13}, Ozeki extended Corso's result for an $
		 \m$-primary ideal in an arbitrary dimension.

Further, studying the properties of special fiber cone, Corso \cite{Corso09} proved that $$f_0(I)\leq e_1(I)-e_1(J)-e_0(I)+\ell(A/I)+\mu(I)-d+1$$ where $f_0(I)$ is the multiplicity of $F_{\m}(I)$ and $J\subseteq I$ is a minimal reduction of $I$.  The boundary case of the above inequality  was studied by Corso-Polini-Vasconcelos \cite{CPV}  and Rossi-Valla  \cite[Chapter 5]{rossi-valla} when $A$ is a Cohen-Macaulay ring. In that case, $F_\m(I)$ need not be Cohen-Macaulay even if the equality
\begin{equation}\label{equality-for-fiber-cone}
f_0(I)= e_1(I)-e_1(J)-e_0(I)+\ell(A/I)+\mu(I)-d+1
	\end{equation}
holds \cite[Example 2.3]{CPV}, however, it remains unmixed \cite[Theorem 2.5]{CPV}. 

Along the similar lines, it is of interest to study the cases when $G(I)$ or $F_\m(I)$ inherits Buchsbaumness from the local ring.  
Our main motivation  is the following result of Corso-Polini-Vasconcelos:
	\begin{remark}\label{remark-intro}\cite[Corollary 2.6 ]{CPV}
		Suppose $A$ is Cohen-Macaulay and the equality  \eqref{equality-for-fiber-cone} holds. Suppose $\depth G(I)\geq d-1.$ Then $F_\m(I)$ is Cohen-Macaulay. 
	\end{remark}
In case of generalized Cohen-Macaulay or Buchsbaum local rings, we aim to obtain results in the same spirit of Remark \ref{remark-intro}. Since it is useful to have some control over the associated graded ring apriori while studying $F_\m(I)$, we consider the equality in \eqref{very-2nd-eqn}
along with \eqref{equality-for-fiber-cone}. Note that  \eqref{very-2nd-eqn}  is stronger than saying $\depth G(I)\geq d-1$ in the set-up of Remark \ref{remark-intro} and interestingly,  
it also enforces the passage of generalized Cohen-Macaulayness (Theorem \ref{t1}) or Buchsbaumness (Theorem \ref{t2}) from $A$ to $G(I)$. 
We prove the following results for generalized Cohen-Macaulay and Buchsbaum local rings, respectively. 
\begin{mainthm}\label{theorem-fiber-m1} (Corollay \ref{corollary_8.5} and \ref{corollay_8.8})
	Let $A$ be a generalized Cohen-Macaulay local ring and  $J=(a_1,\ldots,a_d)\subseteq I$  a reduction of $I$  such that $(a_1,\ldots, \check{a_i},\ldots,a_d):a_i\subseteq I$ for $1\leq i\leq d$. Suppose $J$ is a standard parameter ideal and the equalities in  \eqref{very-2nd-eqn} and \eqref{equality-for-fiber-cone} hold. 
Then,  $F_\m(I)$ is generalized Cohen-Macaulay and geometric regularity of $F_{\m}(I)=\left \lfloor \frac{d}{2}\right \rfloor +1+d$. Moreover, if $\depth A>0$ then  $\depth F_\m(I)=\depth A.$ 
\end{mainthm}
 
	\begin{mainthm}\label{theorem-fiber-m2} (Corollary \ref{theorem-fiber-2(a)})
			Let $A$ be a Buchsbaum local ring with $\depth A\geq d-1$ and   $J=(a_1,\ldots,a_d)\subseteq I$ a reduction of $I$  such that $(a_1,\ldots, \check{a_i},\ldots,a_d):a_i\subseteq I$ for $1\leq i\leq d$. Suppose the equalities in \eqref{very-2nd-eqn} and \eqref{equality-for-fiber-cone} hold. Then 
		$F_\m(I)$ is Buchsbaum. 
		\end{mainthm}
In dimension two, the depth condition may be removed and keeping the rest of hypothesis as in Theorem \ref{theorem-fiber-m2}, one can conclude that $F_{\m}(I)/ \h^0_\U(F_\m (I))$ is Buchsbaum, see Corollary \ref{corollay_8.12}.  Here, $\mathcal{U}:=\m R(I)+R(I)_+$. 

In the above results, the primary role of the equality \eqref{very-2nd-eqn} is its impact on the homological properties of $G(I)$ which can be realized through the following theorems. These are generalizations of earlier results of Ozeki \cite{ozeki12} and \cite{ozeki13} in the general set-up of multiplicative  $I$-good filtration $\mathcal{I}.$  We write $\xi=\m R(J)+R(J)_+$ where $R(J)=\mathop\oplus\limits_{n\geq 0}J^nt^n$. Recall that the {\textit{$\mathds{I}$-invariant}} of an $A$-module $M$ of dimension $d$ is 
	\begin{align*}
		\mathbb{I}(M)= \mathop\sum_{i=0}^{d-1}\binom{d-1}{i}\ell_A(\h_{\m}^i(M)).
	\end{align*}		 
		 \begin{mainthm}\label{t1}(Corollary \ref{Corollary 7.4})
		 	Let $A$ be a generalized Cohen-Macaulay local ring and $J=(a_1,\ldots,a_d)\subseteq I$  a reduction of $\I$  such that $(a_1,\ldots, \check{a_i},\ldots,a_d):a_i\subseteq I$ for $1\leq i\leq d$. Suppose $J$ is a standard parameter ideal  and the equality \begin{equation*}e_1(\I)-e_1(J)=2e_0(\I)-2\ell(A/I_1)-\ell(I_1/(I_2+J))
		 	\end{equation*}  holds. Then $G(\I)$ is generalized Cohen-Macaulay with $$\ell(\h^0_{\xi}(G(\I)))=\ell(\h^0_{\m}(A)),~ \h^i_{\xi}(G(\I))=[\h^i_{\xi}(G(\I))]_{2-i}\cong \h^i_{\m}(A)$$ for $1\leq i\leq d-1$ and $a(G(\I))\leq 2-d.$  Furthermore,
	\begin{enumerate}
\item
 $e_2(\I)=e_1(J)+e_2(J)+e_1(\I)-e_0(\I)+\ell_A(A/I_1)$ \mbox{if } $d\geq 2$;
 \item
 $e_i(\I)=e_{i-2}(J)+2e_{i-1}(J)+e_i(J)$ \mbox{for } $3\leq i\leq d$.
	 \end{enumerate}
		 \end{mainthm}
	 When this is the case, we clearly have $\depth G(\I)=\depth A$ and $\mathbb{I}(G(\I))=\mathbb{I}(A)$. 
\begin{mainthm}\label{t2}(Theorem \ref{m2})
		 		 	Let $A$ be a Buchsbaum local ring and   $J=(a_1,\ldots,a_d)\subseteq I$ a reduction of $\I$  such that $(a_1,\ldots, \check{a_i},\ldots,a_d):a_i\subseteq I$ for $1\leq i\leq d$. Suppose \begin{equation*}e_1(\I)-e_1(J)=2e_0(\I)-2\ell(A/I_1)-\ell(I_1/(I_2+J)).\end{equation*} Then $G(\I)$ is Buchsbaum with $\mathbb{I}(G(\I))=\mathbb{I}(A)$.
		 \end{mainthm}
Indeed, we prove the above results in the set-up of $I$-good filtration of a finitely generated faithful module $M$ and then restrict it to the ring case. Recall that an $I$-filtration  $\M=\{M_n\}_{n\geq 0}$ is a collection of submodules of $M$ satisfying $M_{n+1}\subseteq M_n$, $IM_n\subseteq M_{n+1}$ for all $n\in \mathbb{Z}$. If, in addition, $IM_n=M_{n+1}$ for all $n\gg 0$, then $\M$ is called an $I$-good filtration of $M$.       
      Let $\M=\{M_n\}_{n\geq 0}$ be an $I$-good filtration of $M$. Then the analogue of inequality \eqref{very-1st-eqn} in this module context is: 
	 \begin{equation}\label{very-1st-eqn(a)}
	 	e_1(\M)-e_1^J(M)\geq 2e_0(\M)-2\ell(M/M_1)-\ell(M_1/(M_2+JM)) \end{equation}
	 where $J=(a_1,\ldots,a_d)\subseteq I$ is a minimal $\M$-reduction of $\I$, see
	 \cite[Theorem 2.4]{rossi-valla}.  The distinction is necessary because, in proving the Buchsbaumness of fiber cone $F_\m(I)$, we make crucial use of the filtration $\M_L$, where $L$ is a submodule of $M$ and  $\M_L$ is defined as $
     \M_L:\{ M_0=M, M_{n+1}=I^nL ~\mbox{for}~ n\geq 0\}.$ $\M_L$ is an $I$-good filtration of module $M$. However, in the ring setting where $M=A$ and $L=\m$, the corresponding filtration $\{I_0=A, I_{n+1}=I^n\m ~\text{for}~ n\geq 0\}$ fails to be an multiplicative $I$-good filtration of $A$.
Our methods are inspired by the works \cite{Corso09}, \cite{CPV}, \cite{ozeki12}, \cite{ozeki13}, \cite{rossi-valla}, and \cite{sal20}.

 	We organize the paper into various sections. We begin by introducing our notation in Section \ref{notation}. In Section \ref{Section-Pre}, we discuss preliminary notions and results which are needed subsequently. This section includes a brief discussion on Sally module, $S_J(\M)$ corresponding to the filtration $\M=\{I_nM\}_{n\geq 0}$ and $\M=M_L.$ In Section \ref{Section-inequality}, we mainly discuss some cases when equality holds in \eqref{very-1st-eqn(a)}. In Section \ref{Section-proof-1}, we prove a necessary and sufficient condition for equality in \eqref{very-1st-eqn(a)}. Then we take a detour in Section \ref{Section-detour} to discuss the generalized depth and generalized Cohen-Macaulayness in the set up of module filtration. In Section \ref{Section-main-1}, we complete the proofs of Theorem \ref{t1} and Theorem \ref{t2}. In the last section of the paper, we prove Theorem \ref{theorem-fiber-m1} and Theorem \ref{theorem-fiber-m2}. 
    
\section{Notation}\label{notation}

    
    Throughout this paper, let $\I=\{I_n\}_{n\geq 0}$ be a multiplicative $I$-good filtration of $A$ and $\M=\{M_n\}_{n\geq 0}$ be an $I$-good filtration of $M$ as defined in the previous section. The associated graded module $G(\M)=\mathop{\oplus}\limits_{n\geq 0} M_n/M_{n+1}$ and the Rees module $R(\M)=\mathop{\oplus}\limits_{n\geq 0}^\infty M_nt^n$ are modules over the Rees algebra $R(I)=\mathop{\oplus}\limits_{n\geq 0} I^nt^n$. An ideal $J\subseteq I$ is called an {\em $\M$-reduction of $I$} if $M_{n+1}=JM_{n}$ for $n\gg 0.$ We say that $J$ is a {\em minimal $\M$-reduction of $I$} if it is minimal with respect to inclusion. For a reduction $J$ of $I$, let $$\M_J=\{J^nM\}_{n\geq 0},\text{ an $I$-good filtration of $M$ and }$$  $$R(\M_J)= \mathop{\oplus}\limits_{n\geq 0}^\infty J^nM, \text{ a module over } R(J).$$ 
We regard $R(\M)$ and $R(\M_J)$ as submodules of $M[t]$ over polynomial ring $A[t].$ Let $R(\M)_{\geq m}=\mathop{\oplus}\limits_{n\geq m}M_nt^n$ be the graded submodule of $R(\M).$ The unique homogeneous maximal ideals of $R(I)$ and $R(J)$ are denoted by $\U$ and $\xi$ respectively, i.e., $\U=\m R(I)+R(I)_+$ and $\xi = \m R(J)+R(J)_+.$ We write $$F_I(J)=R(J)/IR(J)=\mathop{\oplus}_{n\geq 0} J^n/IJ^n,$$ $$F_{I_1}(\M_J)=R(\M_J)/I_1R(\M_J)=\mathop{\oplus}\limits_{n\geq 0}J^nM/I_1J^nM \text{ and }$$ $$F_\m(\M_J)=R(\M_J)/\m R(\M_J)=\mathop{\oplus}\limits_{n\geq 0}J^nM/\m J^nM.$$

Let $\overline{A}=A/\h_{\m}^0(A)$, $W=\h_{\m}^0(M)$, and $\overline{M}=M/W$ which is an $\overline{A}$-module. Further, $\overline{\M}=\{ \overline{M_n}=(M_n+W)/W\}_{n\geq 0}$ and $\overline{\M_J}=\{(J^nM+W)/W\}_{n\geq 0}$  are $I\overline{A}$-good filtration of $\overline{M}$.
For  $a\in A$, we write $A^\prime=A/(a)$, $M^\prime=M/aM$, $J^\prime=JA^\prime$, $I^\prime=IA^\prime$ and  $\M^\prime=\{M_n^\prime=(M_n+aM)/aM\}_{n\geq 0}$ which is an $I^\prime$-good filtration of $M^\prime$.  

Let $J=(a_1,\ldots,a_d)$ be a minimal $\M$-reduction of $\I$. 
It follows that $\ell(M/JM)<\infty.$ We consider the following conditions on $a_1,\ldots,a_d$ for $M$: 
	\begin{itemize}
	\item[$(C_0)$] The sequence $a_1,\ldots,a_d$ is a $d$-sequence of $M$ as defined in \cite{adt};
	\item[$(C_1)$] The above sequence is an unconditioned strong d-sequence of $M$;
	\item[$(C_2)$]  $(a_1,\ldots, \check{a_i},\ldots,a_d)M:_Ma_i\subseteq IM$   for $1\leq i\leq d$.
	\item[$(C_3)$] $\depth  M>0$.
\end{itemize}
The above conditions were introduced in \cite{ozeki13} when $M=A$. We note that
all the four conditions $(C_0)$-$(C_3)$ are satisfied if $M$ is Cohen-Macaulay. The condition $(C_1)$ is equivalent to saying that $M$ is generalized Cohen-Macaulay and $J$ is a standard parameter ideal for $M$. The role played by regular sequences in the study of Cohen-Macaulay modules is broadly replaced
by $d$-sequences in the case of generalized Cohen-Macaulay modules. 
A sequence $a_1,\ldots,a_s$ of elements of $A$ is said to be a {\textit{d-sequence}} of $M$ if the equality 
\[q_{i-1}M:_M a_ia_j=q_{i-1}M:_M a_j\] holds for $1\leq i\leq j\leq s$ where 
$q_{i-1}=(a_1,\ldots,a_{i-1})$ and $q_0=(0)$. It is said to be an {\textit{unconditioned d-sequence}} if it is a $d$-sequence in any order. Moreover, we say that $a_1,\ldots,a_s$ is an {\textit{unconditioned strong $d$-sequence}} (u.s.d-sequence) if $a_1^{n_1},\ldots,a_s^{n_s}$ is an  unconditioned $d$-sequence for  all integers $n_1,\ldots,n_s>0$. 
Recall that a module is Buchsbaum if and only if every system of parameters of $M$ is u.s.d-sequence of $M$. So,  $(C_0)$ and $(C_1)$ are satisfied by every system of parameters of $M$ if $M$ is a Buchsbaum module.

\begin{remark}\label{remark-nota}
Suppose  $(C_1)$ is satisfied by $(a_1,\ldots,a_d)$ for $M$. Then $a_i W\subseteq JW\subseteq JM\cap W=0$,  see \cite[Corollary 2.3]{trung-gcm}. This gives 
$ W\subseteq (a_1,\ldots, \check{a_i},\ldots,a_d)M:_Ma_i$ for $1\leq i\leq d.$
\end{remark}

By $\mu(M)$, we denote the minimal number of generators of $M$ and $e_i(\M)$, $1\leq i\leq d$ are the Hilbert coefficients of $\M$, i.e., the unique integers such that $\ell(M/M_n)=\sum\limits_{i=0}^d(-1)^ie_i(\M)\binom{n+d-i-1}{d-i}$ for $n\gg 0.$ When $\M$ is the $I$-adic filtration $\{I^nM\}_{n\geq 0}$, we write $e_i^I(M)$ in place of $e_i(\M)$. 

We assume that the residue field $A/\m$ is infinite. For convenience, we highlight that for most of our results, as in \cite{rossi-valla}, $\M$ is either $$\M_\I=\{I_nM\}_{n\geq 0}\text{ or }$$ $$\M_L=\{M_0=M, ~M_{n+1}=I^nL ~\text{ for } ~n\geq 0\} \text{ where } L=\m M 
\text{ and }$$
through out the paper, $J=(a_1,\ldots,a_d)\subseteq I$ denote either a minimal $\M_\I$-reduction of $I$ or a minimal $\M_L$-reduction of $I$ which will be clear from the context.

\section{Preliminaries}\label{Section-Pre}
In this section, we present some auxiliary results that will be used in subsequent sections. We discuss most of the proofs for the reader's convenience and mention the references for $I$-adic cases of some results.
\subsection{Sally module of filtration}
To study blow-up algebras, Vasconcelos \cite{Vas94} introduced the notion of Sally modules in the ring case. 
Further, Rossi-Valla in \cite[Chapter 6]{rossi-valla}, extended the definition for an $I$-good filtration of a module which we recall below. The Sally module, $S_J(\M)$ corresponding to the filtration $\M_\I$ or $\M_L$ can be defined as the cokernel of the following maps of graded $R(J)$-modules respectively:
\begin{align}\label{eq_1}
 &0\longrightarrow I_1R(\M_J)\longrightarrow \R(\M)_{+}(+1)\longrightarrow S_J(\M)\longrightarrow 0 \\
 &0\longrightarrow \m R(\M_J)\longrightarrow \R(\M)_{+}(+1)\longrightarrow S_J(\M)\longrightarrow 0
\end{align}
where $J=(a_1,\ldots,a_d)\subseteq I$ is a minimal $\M_\I$-reduction or $\M_L$-reduction  of $I$ respectively. Therefore, for both the filtration, Sally module is a finitely generated graded $R(J)$-module and $\dim_{R(J)} S_J(\M)\leq d$.  

\begin{lemma}\label{lemma-3.1(a)} Let $M$ be a finitely generated $A$-module of dimension $d$.  Suppose $\M=\M_\I$ and $J\subsetneq I_1$ or $\M=\M_L$ and $J\subsetneq \m.$ 
Then there exists an exact sequence of graded $R(J)$-modules
$$
    R(\M_J)(-1)^\mu\xrightarrow{\phi}R(\mathcal{M})/ R(\M_J)\to S_J(\M)(-1)\to 0
    $$ 
    where  $\mu=\mu(\m \setminus J)$ when $\M=\M_L$ and  $\mu=\mu(I_1\setminus J)$ when $\M=\M_\I.$ 
\end{lemma}
\begin{proof}
 Suppose $\M=\M_L$ and  $\m =J+(y_1,\ldots,y_\mu)$, where $y_1,\ldots,y_\mu \in \m$. Consider the graded map 
	$$R(\M_J)(-1)^\mu\xrightarrow{\phi} R(\mathcal{M})/R(\M_J)$$
	defined as $\phi(\alpha_1,\ldots,\alpha_\mu )=\overline{\mathop\sum\limits_{i=1}^\mu y_i \alpha_i }\in R(\mathcal{M})/R(\M_J)$, where $\alpha_i \in J^{n-1}M $ and $y_1,\ldots,y_{\mu}\in \m $. Then
	$[\coker \phi]_{n}=\m I^{n-1}M/(J^nM+J^{n-1}\m M)=\m I^{n-1}M/J^{n-1}\m M=(S_J(\M)(-1))_{n}$ for $n\geq 1$ and
	$[\coker \phi]_{0}=0$. Therefore,
	$\coker \phi\cong S_J(\M)(-1)$ as graded $R(J)$-module.

    Now assume $\M=\M_\I$ and $I_1=J+(y_1,\ldots,y_\mu)$, where $y_1,\ldots,y_\mu \in I_1$. As earlier, define the graded map $R(\M_J)(-1)^\mu\xrightarrow{\phi} R(\mathcal{M})/R(\M_J)$ by  $\phi(\alpha_1,\ldots,\alpha_\mu )=\overline{\mathop\sum\limits_{i=1}^\mu y_i \alpha_i}$. Then $[\coker \phi]_{n}=I_{n}M/(J^nM+J^{n-1}I_1M)=I_{n}M/J^{n-1}I_1M=(S_J(\M)(-1))_{n}$ for $n\geq 1$ and
	$[\coker \phi]_{0}=0$ which gives 
	$\coker \phi\cong S_J(\M)(-1)$ as graded $R(J)$-module.
\end{proof}
\begin{lemma}\label{lemma-2.2(a)}
	Let $M$ be a finitely generated flat $A$-module of dimension $d>0$. Suppose $(C_0)$ and $(C_2)$ conditions are satisfied by the sequence $a_1,\ldots,a_d$ for both $A$ and $M$. Then 
    \begin{enumerate}
        \item $R(\M_J)/I_1R(\M_J)=\F_{I_1}(\M_J)\cong (M/I_1M)[X_1,\ldots,X_d]$ as graded $R(J)$-module,
        \item $R(J)/IR(J)=\F_I(J) \cong A/I[X_1,\ldots,X_d]$ as graded $A$-algebra and 
        \item $R(\M_J)/\m R(\M_J)=\F_{\m}(\M_J)\cong (M/\m M)[X_1,\ldots,X_d]$ as graded $R(J)$-module. 	
    \end{enumerate}  
	In particular, $\F_I(J)$ is Cohen-Macaulay ring of dimension $d$, $\F_{I_1}(\M_J)$ and  $\F_\m(\M_J)$ are Cohen-Macaulay $R(J)$-modules of dimension $d$.
\end{lemma}
\begin{proof}
    Since $M$ is a flat $A$-module, we have $R(\M_J)\cong R(J) \otimes_A M.$ By condition $(C_0)$, we get that $R(J)\cong Sym_A(J)$, where $Sym_A(J)$ denotes the symmetric algebra of $J$. Notice that 
\begin{eqnarray*}
\frac{R(\M_J)}{I_1R(\M_J)}\cong \frac{A}{I_1}\otimes_A R(J)\otimes_A M &\cong&  \frac{A}{I_1} \otimes_A Sym_A(J)\otimes_A M\\
&\cong& Sym_{A/I_1}\Big(\frac{A}{I_1} \otimes_A J\Big)\otimes_A M\\
&=&Sym_{A/I_1}\Big(\frac{J}{I_1J}\Big)\otimes_A M \text{ as graded $A/I_1$-module and } 
\end{eqnarray*}
\begin{eqnarray*}
    \frac{R(\M_J)}{\m R(\M_J)}\cong \frac{A}{\m}\otimes_A R(J)\otimes_A M &\cong&  \frac{A}{\m} \otimes Sym_A(J)\otimes_A M\\
    &\cong& Sym_{A/\m}\Big(\frac{A}{\m} \otimes_A J\Big)\otimes_A M\\&=&Sym_{A/\m}\Big(\frac{J}{\m J}\Big)\otimes_A M  \text{ as graded $A/\m$-module. }
\end{eqnarray*}

By condition $(C_2)$, $J/I_1J$ and $J/\m J$ are free  of rank $d$ over $A/I_1$ and $A/\m$ respectively. Hence $Sym_{A/I_1}(J/I_1J)\cong A/I_1[X_1,\ldots,X_d]$ and $Sym_{A/\m}(J/\m J)\cong A/\m[X_1,\ldots,X_d]$.   Therefore, we have 
\begin{align*}
R(\M_J)/I_1R(\M_J)&\cong A/I_1[x_1,\ldots,X_d]\otimes_A M \cong M/I_1M[X_1,\ldots, X_d] \text{ and }\\
R(\M_J)/\m R(\M_J)&\cong A/\m [x_1,\ldots,X_d]\otimes M \cong M/\m M[X_1,\ldots, X_d]\qedhere
\end{align*}
\end{proof}
\begin{lemma}\cite[Corollary 6.2]{rossi-valla}\label{lemma-2.3(a)}
    Let $M$ be a Cohen-Macaulay $A$-module of dimension $d$ and $L$ be a submodule of $M$. Then $S_J(\M_L)=0$ or $\dim S_J(\M_L)=d.$

\end{lemma}
To determine the dimension of $S_J(\M)$ when $\M=\M_{\I}$, we state the following lemma. 
 \begin{lemma}\label{lemma-2.3(b)} Suppose $\depth A>0$, $M$ is a finitely generated faithful $A$-module of dimension $d$ and $\depth M>0.$ Suppose that the conditions $(C_0)$ and $(C_2)$  are satisfied by $a_1,\ldots,a_d$ for both $A$ and $M$.  If $R(\M_J)$ satisfies Serre's condition $(S_2)$, then $\ass_{R(J)}(S_J(\M_\I))\subseteq \{\m R(J)\}$. Consequently, $\dim_{R(J)}(S_J(\M_\I))=d$ if $S_J(\M_\I)\neq 0$.
 \end{lemma}
 \begin{proof}
 	Let $P\in\ass_{R(J)}(S_J(\M_\I))$. Since $\m^l S_J(\M_\I)=0$ for $l\gg 0$, we have $\p=\m R(J)\subseteq P.$ Suppose $\p\neq P$. 
 	Note that $\depth_{R(J)_P} (R(\M_\I)_{+}(+1))_P\geq 1$
 	since $a_1\in \p\subseteq P$
 	is a non-zero-divisor on $R(\M_\I)_{+}(+1)=\mathop\oplus\limits_{j\geq 1}I_jMt^{j-1}$. Applying the depth Lemma on the exact sequence
 	$$0\longrightarrow (I_1 R(\M_J))_P\longrightarrow (R(\M_\I)_{+}(+1))_P\longrightarrow (S_J(\M_\I))_P\longrightarrow 0,$$
 	of graded $R(J)_P$-modules. We get $$\depth_{R(J)_P} (I_1 R(\M_J))_P\geq \min\{ \depth_{R(J)_P} (R(\M_\I)_{+}(+1))_P,\depth_{R(J)_P} (S_J(\M_\I))_P+1\}\geq 1$$ and 
 	$\depth_{R(J)_P} (S_J(\M_\I))_P=0\geq \min\{\depth_{R(J)_P} (R(\M_\I)_{+}(+1))_P,\depth_{R(J)_P} (I_1 R(\M_J))_P-1\}=\depth_{R(J)_P} (I_1 R(\M_J))_P-1$. Hence $\depth_{R(J)_P} (I_1 R(\M_J))_P=1$.
 	Now consider the exact sequence
 	$$0\longrightarrow (I_1 R(\M_J))_P\longrightarrow R(\M_J)_P\longrightarrow (R(\M_J)/I_1 R(\M_J))_P\longrightarrow 0.$$
 	Since $\p\neq P$, $\height~P\geq 2$.
 	So $\dim R(J)_P\geq 2$. Since $M$ is faithful $A$-module, hence $R(\M_J)$ is faithful $R(J)$-module. Therefore $\dim R(\M_J)_P\geq 2$ which implies that $\depth R(\M_J)_p\geq \min\{2,\dim R(\M_J)_P\}=2$ as $R(\M_J)$ satisfies $(S_2)$. 
 	Again by depth lemma, $1=\depth_{R(J)_P} (I_1 R(\M_J))_P\geq \min\{\depth_{R(J)_P} R(\M_J)_P,\depth_{R(J)_P} (R(\M_J)/I_1 R(\M_J))_P+1 \}=\depth_{R(J)_P} (R(\M_J)/I_1 R(\M_J))_P+1$ which implies
 	$\depth_{R(J)_P} (R(\M_J)/I_1 R(\M_J))_P=0.$ By Lemma \ref{lemma-2.2(a)}, $R(\M_J)/I_1 R(\M_J)$ is Cohen-Macaulay which means $\dim (R(\M_J)/I_1 R(\M_J))_P=\depth_{R(J)_P} (R(\M_J)/I_1 R(\M_J))_P=0.$ Thus 
 	$P\in\min_{R(J)}(R(\M_J)_P/I_1 R(\M_J)_P)=\{\m R(J)\}$ which is a contradiction.
 \end{proof}
The Hilbert coefficients $e_i(S_J(\M))$ of the Sally module $S_J(\M)$ and the coefficients $e_i(\M)$ of $\M$ are related through the following relations. Here $\M=\M_L$ or $\M=\M_\I.$  We have
\[\ell(M/M_{n+1})=\ell(M/J^n M_1)-\ell(S_J(\M)_n)= \ell(M/J^nM)+\ell(J^nM/J^n M_1)-\ell(S_J(\M)_n)\]
for $n\geq 0$. 
\begin{proposition}\label{proposition:p1} Let $M$ be a finitely generated faithful $A$-module and $\M=\M_L$ or $ \M=\M_\I.$  Then
	\begin{equation*}
		e_1(\M)\geq  \begin{cases} e_1^J(M)+e_0^J(M)-\ell_A(M/M_1)+e_0(S_J(\M)) &\text{ if } \dim S_J(\M)=d.\\
			e_1^J(M)+e_0^J(M)-\ell_A(M/M_1) & \text{ if } \dim S_J(\M)\leq d-1.
		\end{cases}
	\end{equation*}
\end{proposition}
\begin{proof}
	For $n\geq 0$, when $M_1=I_1M$, we have $J^nM/J^n M_1= (R(\M_J)/I_1 R(\M_J))_n$ and $R(\M_J)/I_1 R(\M_J)$ is a homomorphic image of $M/M_1[X_1,\ldots,X_d]$ or when $M_1=\m M$, we have $J^nM/J^n M_1= (R(\M_J)/\m R(\M_J))_n$ and $R(\M_J)/\m R(\M_J)$ is a homomorphic image of $M/M_1[X_1,\ldots,X_d]$. So for all $n\geq 0$, 
	\begin{equation}
		\ell(J^nM/J^{n} M_1)\leq \ell(M/M_1)\binom{n+d-1}{d-1}.\label{equation:step-hilbFunfibercone}
	\end{equation}
	There exists an integer $N$ such that for all $n\geq N$,
	\begin{align*}
		\ell(M/J^nM)&=\mathop\sum\limits_{i=0}^d (-1)^ie_i^J(M)\binom{n-1+d-i}{d-i}\\
		&=e_0^J(M)\binom{n+d}{d}-(e_0^J(M)+e_1^J(M))\binom{n+d-1}{d-1}\\
		&\qquad+\mathop\sum\limits_{i=2}^d(-1)^i(e_{i-1}^J(M)+e_i^J(M))\binom{n+d-i}{d-i}.
	\end{align*}
	Thus, for all $n\geq N$,
	\begin{align*}
		\begin{split}
			\ell(M/M_{n+1})+\ell(S_J(\M)_n)&=\ell(M/J^nM)+\ell(J^nM/J^n M_1)\\
			&\leq e_0^J(M)\binom{n+d}{d}-(e_0^J(M)+e_1^J(M))\binom{n+d-1}{d-1}+\mathop\sum\limits_{i=2}^d(-1)^i\\
			&\qquad\big(e_{i-1}^J(M)+ e_i^J(M)\big)\binom{n+d-i}{d-i}+\ell(M/M_1)\binom{n+d-1}{d-1}\\
			&\leq e_0^J(M)\binom{n+d}{d}-\Big\{ e_0^J(M)+e_1^J(M)-\ell(M/M_1)\Big\}\binom{n+d-1}{d-1}\\
			&\qquad+\mathop\sum\limits_{i=2}^d(-1)^i(e_{i-1}^J(M)+e_i^J(M))\binom{n+d-i}{d-i} 
		\end{split}
	\end{align*}
	Now, the conclusion  follows on comparing the coefficients.
\end{proof}
\begin{proposition}\label{proposition-2.4}    Let $M$ be a finitely generated faithful flat $A$-module and $\M=\M_L$ or $\M=\M_\I.$ Suppose $(C_0)$ and $(C_2)$ conditions are satisfied by $a_1,\ldots,a_d$ for both $A$ and $M$. Then,  for all $n\geq 0,$ 
	\begin{align}\label{newlabel-eq-1}
		\begin{split}
			\ell_A(M/M_{n+1})&= e_0^J(M)\binom{n+d}{d}-\Big\{ e_0^J(M)+e_1^J(M)-\ell(M/M_1)\Big\}\binom{n+d-1}{d-1}\\
			&\qquad+\mathop\sum\limits_{i=2}^d(-1)^i(e_{i-1}^J(M)+e_i^J(M))\binom{n+d-i}{d-i}-\ell_A(S_J(\M)_{n}).
		\end{split}
	\end{align}
\end{proposition}
\begin{proof} Suppose $(C_0)$ and $(C_2)$ condition are satisfied by $a_1,\ldots,a_d$ for both $A$ and $M$. Then by \cite[Proposition 4.1]{trung-abs}, $N=0$ in the proof of Proposition \ref{proposition:p1}. We have $J^nM/J^n M_1= (R(\M_J)/I_1 R(\M_J))_n$ when $\M=\M_\I$ and $J^nM/J^n M_1= (R(\M_J)/\m R(\M_J))_n$ when $\M=\M_L$. By Lemma \ref{lemma-2.2(a)}, it is isomorphic to $M/M_1[X_1,\ldots,X_d]$.  Hence, equality holds in
	\eqref{equation:step-hilbFunfibercone} for all $n\geq 0.$
\end{proof}
In \eqref{newlabel-eq-1}, we may write
$$\ell_A((S_J(\M)_n)=e_0(S_J(\M))\binom{n+s-1}{s-1}-e_1(S_J(\M))\binom{n+s-2}{s-2}+\ldots+(-1)^{s-1}e_{s-1}(S_J(\M))$$
for $n\gg 0$ where $s=\dim S_J(\M)$. Then, comparing the coefficients of both sides, we get the following corollary.
\begin{corollary}\label{corr-2.5}
Let the hypothesis be the same as in Proposition \ref{proposition-2.4}. Let $\dim S_J(\M)=s$. 
	\begin{enumerate}[label=(\roman*)]
		\item Suppose $s=d$. Then 
		\begin{enumerate}
			\item $e_1(\M)=e_0^J(M)+e_1^J(M)-\ell(M/M_1)+e_0(S_J(\M))$ and 
			\item $e_i(\M)=e_{i-1}^J(M)+e_i^J(M)+e_{i-1}(S_J(\M))$ for $2\leq i\leq d$.
		\end{enumerate}
		\item Suppose $s<d$. Then 
		\begin{enumerate}
			\item $e_1(\M)=e_0^J(M)+e_1^J(M)-\ell(M/M_1)$,
			\item $e_i(\M)=e_{i-1}^J(M)+e_i^J(M)$ for $2\leq i\leq d-s$ and
			\item $e_i(\M)=e_{i-1}^J(M)+e_i^J(M)+(-1)^{d-s}e_{i-d+s-1}(S_J(\M))$ for all $d-s+1\leq i\leq d.$
		\end{enumerate}
	\end{enumerate}
\end{corollary}
\subsection{Reduction steps} In the theory of Hilbert polynomials, it is very useful to reduce a result (i) to the case when $\depth M>0$ and (ii) to a lower-dimensional case for applying induction. The following results allow us to do both of these reductions in most of our results.

\begin{lemma}\label{lemma-reduction-steps}
	\begin{enumerate}[label=(\roman*)]
		\item  If $(C_1)$ and $(C_2)$ conditions are satisfied by $a_1,\ldots,a_d$ for $M$, then $(C_1)$ and $(C_2)$ conditions are satisfied by $a_1,\ldots,a_d$ for $\overline{M}$. 
		
		\item \label{lemma-4.1} Suppose that $d\geq 2$. Let $a\in J\setminus \m J$ be a superficial element for $\M$ and $\M_J$.  Then  $2e_0(\M)-e_1(\M)+e_1^J(M)=2\ell_A(M/M_1)+\ell_A(M_1/(M_2+JM))$ if and only if
		$2e_0(\M^\prime)-e_1(\M^\prime)+e_1^J(M^\prime)=2\ell_{A}(M^\prime/{M}^\prime_1)+\ell_{A}({M}^\prime_1/({M}^\prime_2+JM^\prime)).$
		
		\item \label{lemma-4.2} $2e_0(\M)-e_1(\M)+e_1^J(M)=2\ell_A(M/M_1)+\ell_A(M_1/(M_2+JM))$ if and only if  
		$2e_0(\overline{\M})-e_1(\overline{\M})+e_1^J(\overline{M})=2\ell_A(\overline{M}/\overline{M_1} )+\ell_A(\overline{M_1} /(\overline{M_2} +J\overline{M}))$ and $W\subseteq M_2+JM$.
	\end{enumerate}
\end{lemma}
\begin{proof}
	\begin{enumerate}[label=(\roman*)]
		\item
		Suppose $a_1,\ldots,a_d$ satisfies  $(C_1)$ for $M$. Then it follows from \cite[Theorem 2.1]{trung-gcm} and \cite[Lemma 1.6]{trung-gcm} that $\overline{M}$ is generalized Cohen-Macaulay with $J\overline{A}$ a standard parameter ideal in $\overline{M}$ and $JM\cap W=0$ from \cite[Corollary 2.3]{trung-gcm}. Now let  $(a_1,\ldots,\check{a_i},\ldots,a_d)M:_M a_i\subseteq IM$ and $r+W\in (a_1,\ldots,\check{a_i},\ldots,a_d)\overline{M}:_{\overline{M}} a_i$.  Then $ra_i\in (a_1,\ldots,\check{a_i},\ldots,a_d)M+W$. Suppose $ra_i=s+w_0$ with $s\in (a_1,\ldots,\check{a_i},\ldots,a_d)M$ and $w_0\in W$. Then
		$w_0\in JM\cap W=0$ which implies that $ r\in (a_1,\ldots,\check{a_i},\ldots,a_d)M: a_i\subseteq IM.$ Therefore,   $(a_1,\ldots,\check{a_i},\ldots,a_d)\overline{M}:_{\overline{M}} a_i \subseteq I\overline{M}.$
        \vspace{0.25cm}
        
		\item We have $e_i(\M)=e_i(\M^\prime)$, $e_i^J(M)=e_i^J(M^\prime)$ for $0\leq i\leq d-2$, $e_{d-1}(\M)=e_{d-1}(\M^\prime)+(-1)^d\ell(0:_M a)$ and $e_{d-1}^J(M)=e_{d-1}^J(M^\prime)+(-1)^d\ell(0:_M a)$. Therefore, $2e_0(\M)-e_1(\M)+e_1^J(M)= 2e_0(\M^\prime)-e_1(\M^\prime)+e_1^J(M^\prime)$. Further, $2\ell_A(M/M_1)+\ell_A(M_1/(M_2+JM))=2\ell_{A}(M^\prime/M_1^\prime)+\ell_{A}(M_1^\prime/(M_2^\prime+JM^\prime))$.

         \vspace{0.25cm}
		\item 
		We have $e_i(\M)=e_i(\overline{\M})$ and $e_i^J(M)=e_i^J(\overline{M})$ for $0\leq i\leq d-1$, $e_d(\M)=e_d(\overline{\M})+(-1)^d\ell_A(W)$ and $e_d^J(M)=e_d^J(\overline{M})+(-1)^d\ell_A(W)$. Thus,
		\begin{align*}
			2e_0(\M)-e_1(\M)+e_1^J(M)&= 2e_0(\overline{\M})-e_1(\overline{\M})+e_1^J(\overline{M})\\
			&\leq 2\ell_A(\overline{M}/\overline{M_1})+\ell_A(\overline{M_1}/(\overline{M_2}+J\overline{M})) \mbox{ by \eqref {very-1st-eqn(a)}}\\
			&\leq 2\ell_A(M/M_1)+\ell_A(M_1/(M_2+JM)).
		\end{align*}
		Therefore, $2e_0(\M)-e_1(\M)+e_1^J(M)= 2\ell_A(M/M_1)+\ell_A(M_1/(M_2+JM))$ if and only if $2e_0(\overline{\M})-e_1(\overline{\M})+e_1^J(\overline{M})=2\ell_A(\overline{M}/\overline{M_1})+\ell_A(\overline{M_1}/(\overline{M_2}+J\overline{M})),$ $\ell_A(\overline{M_1}/(\overline{M_2}+J\overline{M}))=\ell_A(M_1/(M_2+JM))$  and $\ell_A(\overline{M}/\overline{M_1})=\ell_A(M/M_1)$. 
		Note that the last two equalities hold if and only if $W\subseteq M_2+JM$. \qedhere
	\end{enumerate}
\end{proof}
We end this section with the following remark. 
\begin{remark}\label{remark-I1=Q-case}
	Suppose $e_1(\M)-e_1^J(M)=2(e_0^J(M)-\ell(M/JM)).$ Then the following conditions hold
    :
	\begin{enumerate}[label=(\roman*)]
		\item  $e_1(\M)=e_1^J(M)$,
		\item $M$ is Cohen-Macaulay $A$-module and 
		\item $\M=\M_J.$
	\end{enumerate}
\end{remark}
\begin{proof}
	We have $e_0(\M)-\ell(M/M_1)\leq e_1(\M)-e_1^J(M)=2(e_0^J(M)-\ell(M/JM))$ which gives 
	$e_0^J(M)-2\ell(M/JM)\geq -\ell(M/M_1)\Rightarrow e_0^J(M)-\ell(M/JM)\geq \ell(M_1/JM)\geq 0.$ On the other hand, $e_0^J(M)\leq \ell(M/JM)$. Therefore, we get $e_0^J(M)=\ell(M/JM)$. 
	Thus $M$ is Cohen-Macaulay $A$-module, see \cite[Corollary 4.7.11]{bruns-herzog} and $e_1(\M)=e_1^J(M)=0.$ By \cite[Theorem 2.7]{rossi-valla} $G(\M)$ is Cohen-Macaulay $A$-module and $M_n\subseteq JM$ for all $n.$ Hence by Valabrega-Valla criteria $M_n=M_n\cap JM= JM_{n-1}$ for all $n\geq 1.$ This gives that $\M=\M_J.$
\end{proof}
\section{\texorpdfstring
{The boundary condition of inequality \eqref{very-1st-eqn(a)}}
{The boundary condition of inequality (1)}
}\label{Section-inequality}
The main result of this section is Theorem \ref{lemma-3.3} which provides a way to discuss necessary conditions for the equality 
 $ e_1(\M)-e_1^J(M)= e_0^J(M)-  2\ell_A(M/M_1)-\ell_A(M_1/(M_2+JM))$
in case of filtrations $\M=\M_L$ and $\M=\M_\I$, see Corollary \ref{corollary-3.6}. 

Now onwards, we restrict ourselves to write $\M=\{M_n\}_{n\geq 0}$ for either the filtration $\M_\I$ or $\M_L$. 
For a minimal $\M$-reduction $J$ of $I$ , we consider the canonical graded $R(J)$-module homomorphisms: $$\F_{I_1}(\M_J)\to G(\M_\I)=R(\M_\I)/\R(\M_\I)_{+}(1) \text{ and }$$  $$\F_{\m}(\M_J)\to G(\M_L)=R(\M_L)/\R(\M_L)_{+}(1)$$ which give  exact sequences 
	\begin{eqnarray} \label{equation:step1}
		0\to K^{(2)}(\M_\I)\to \F_{I_1}(\M_J)\to G(\M_\I)\to R(\M_\I)/(R(\M_\I)_{+}(1)+R(\M_J))\to 0 \text { and }\\
       \label{equation:step1(a)} 0\to K^{(2)}(\M_L)\to \F_{m}(\M_J)\to G(\M_L)\to R(\M_L)/(R(\M_L)_{+}(1)+R(\M_J))\to 0
	\end{eqnarray}    
	of graded $R(J)$-modules where $K^{(2)}(\M_\I)$ and $K^{(2)}(\M_L)$ denote the respective kernels.  Now, let \begin{equation}
    \mu= \begin{cases}
\mu(I_1\setminus J) & \mbox{ if } \M=\M_\I \\  
\mu(\m/J) & \mbox{ if } \M=\M_L 
    \end{cases}
    \end{equation}
and $y_1,\ldots,y_{\mu}$, elements of $A$, such that $I_1=J+(y_1,\ldots,y_\mu)$ when $\M=\M_\I$ and $\m=J+(y_1,\ldots,y_\mu)$ when $\M=\M_L$.  Consider the graded homomorphism of graded $R(J)$-modules $$((R(\M_J)/IR(\M_J))(-1))^\mu\xrightarrow{\phi} R(\M)/(R(\M)_{+}(1)+R(\M_J))$$ defined as
	$$\phi(\alpha_1,\ldots,\alpha_\mu )=\overline{\mathop\sum_{i=1}^\mu y_i\alpha_i  }\in R(\M)/(R(\M)_{+}(1)+R(\M_J))$$
    where $\alpha_i\in R(\M_J)/I R(\M_J)$ for $1\leq i\leq \mu.$ Then 
$[\coker \phi]_n=\frac{M_n/(M_{n+1}+J^nM)}{( J^{n-1}M_1+M_{n+1})/(M_{n+1}+J^nM)}\cong M_n/(M_{n+1}+ J^{n-1}M_1)=(S_J(\M)/\L(\M))_{n-1}$ where $\L(\M)=\mathop\oplus\limits_{n\geq 1}\frac{M_{n+2}+J^{n}M_1}{ J^{n}M_1}$.  Thus 
	$\coker\phi\cong (S_J(\M)/\L(\M))(-1)$ and we get the exact sequence
	\begin{equation}\label{equation:step2}
    \begin{aligned}
        0\to \ker\phi&\to ((R(\M_J)/IR(\M_J)) (-1))^\mu\xrightarrow{\phi} R(\M)/(R(\M)_+(1)+R(\M_J))\\
        &\to (S_J(\M)/\L(\M))(-1)\to 0.
    \end{aligned}
	\end{equation}
Now we state the main theorem of this section. For the proof, we borrow the ideas from \cite{ozeki13}.  
 \begin{theorem}\label{lemma-3.3} 
	 Let $M$ be a finitely generated faithful flat $A$-module, $\M=\{M_n\}_{n\geq 0}$ be one of  the filtration $\M_\I$ or $\M_L$  and $J=(a_1,\ldots,a_d)$ be a minimal $\M$-reduction of $I$. We assume $J\subsetneq I_1$ in case of $\M=\M_\I$ and $J\subsetneq \m$  in case of $\M=\M_L$. Suppose the conditions $(C_0)$ and $(C_2)$ are satisfied by $a_1,\ldots,a_d$ for both $A$ and $M$. Then 
    \begin{align*}		\ell_A(M_n/M_{n+1})&=\Big \{\ell_A(M/M_1)+\ell_A(M_1/(M_2+JM))\Big\}\binom{n+d-1}{d-1}
		-\ell(M_1/({M_2}+JM))\\&\qquad \binom{n+d-2}{d-2}
		 +\ell_A(S_J(\M)_{n-1})-\ell_A([K^{(1)}(\M)]_{n})-\ell_A([K^{(2)}(\M)]_{n})-\\&\qquad\ell_A([K^{(3)}(\M)]_{n-1})
	\end{align*}
	for all $n\geq 0$ and for some   graded $R(J)$-modules $K^{(i)}(\M)$, $1\leq i\leq 3$. Furthermore, 
	$$
		e_1(\M)-e_1^J(M)= 2 e_0^J(M)-  2\ell_A(M/M_1)-\ell_A(M_1/(M_2+JM))+\mathop\sum\limits_{i\in\Gamma}e_0(K^{(i)}(\M))
	$$
    where $\Gamma=\{i\mid 1\leq i\leq 3, ~\dim_{R(J)} K^{(i)}(\M)=d\}.$
\end{theorem}
\begin{proof}
First, we consider the following exact sequence induced from \eqref{equation:step2},
	$$0\to [\ker\phi]_1\to {[(R(\M_J)/IR(\M_J))}^\mu]_0\to [R(\M)/(R(\M)_+(1)+R(\M_J))]_1\to 0$$
	and tensor it with $\F_I(J)$ to get the following commutative diagram:
	{\tiny\[
	\xymatrixrowsep{20mm}\xymatrixcolsep{3mm}
	\xymatrix{
		& ([\ker \phi]_1\otimes \F_I(J))(-1) \ar[r] \ar[d]_{\psi_1}  & ({[(R(\M_J)/IR(\M_J))}^\mu]_0\otimes \F_I(J))(-1) \ar[r] \ar[d]^{\psi_2} & ([R(\M)/(R(\M)_+(1)+R(\M_J))]_1 \otimes \F_I(J))(-1) \ar[r]\ar[d]^{\psi_3} & 0 \\
		0 \ar[r] & \ker \phi\ar[r]^{f}  &  [{(R(\M_J)/IR(\M_J)(-1))}^\mu] \ar[r]  & \im{\phi} \ar[r] & 0 
	}
	\]
    }
	of graded $R(J)$-modules where $\psi_2$ is bijective. So, we have an exact sequence \begin{equation}\label{equation:step3}
		0\to K^{(1)}(\M)\to ([R(\M)/(R(\M)_+(1)+R(\M_J))]_1\otimes\F_I(J))(-1)\xrightarrow{\psi_3}\im  \phi\to 0
	\end{equation}
	of graded $R(J)$-modules where $K^{(1)}(\M)=\ker\psi_3$. We put $K^{(3)}(\M)=\L(\M)$ for convenience of notation. Note that $K^{(1)}(\M)$ is a submodule of $ ([R(\M)/(R(\M)_+(1)+R(\M_J))]_1\otimes \F_I(J))(-1)$, $K^{(2)}(\M)$ is a submodule of $ \F_{I_1}(\M_J)$ or $\F_{\m}(\M_J)$  and $K^{(3)}(\M)$ is a submodule of $ S_J(\M)$. So, $\dim K^{(i)}(\M)\leq d$ for $i=1,2,3$.   

By Lemma \ref{lemma-2.2(a)},   $\F_{I_1}(\M_J)\cong M/I_1M[X_1,\ldots,X_d]$ and $\F_{\m}(\M_J)\cong M/\m M[X_1,\ldots,X_d].$ Therefore,  $\ell_A([\F_{I_1}(\M_J)]_n)$ and $\ell_A([\F_{\m}(\M_J)]_n)$ are equal to  $\ell_A(M/M_1)\binom{n+d-1}{d-1}$ for all $n\geq 0.$

	From \eqref{equation:step1}, \eqref{equation:step1(a)}, \eqref{equation:step2} and \eqref{equation:step3}, we have that 
	\begin{align}
		\ell_A(G(\M)_n)&=\ell_A\Big(\frac{M}{M_1}\Big)\binom{n+d-1}{d-1}+\ell_A\Big(\Big[\frac{R(\M)}{R(\M)_+(1)+R(\M_J)}\Big]_n\Big)-\ell_A([K^{(2)}(\M)]_n)\nonumber\\
		&= \ell_A\Big(\frac{M}{M_1}\Big)\binom{n+d-1}{d-1}+\{\ell_A([\im \phi]_n)+\ell_A\Big(\Big[\frac{S_J(\M)}{\L(\M)}\Big]_{n-1}\Big)\}-\ell_A([K^{(2)}(\M)]_n)\nonumber\\
		&=\ell_A\Big(\frac{M}{M_1}\Big)\binom{n+d-1}{d-1}+\{ \ell_A\Big(\Big[\Big[\frac{R(\M)}{R(\M)_+(1)+R(\M_J)}\Big]_1\otimes \F_I(J)\Big]_{n-1}\Big)-\nonumber\\
		&\qquad\ell_A([K^{(1)}(\M)]_n)\} +\ell_A(S_J(\M)_{n-1})-\ell_A([K^{(3)}(\M)]_{n-1})-\ell_A([K^{(2)}(\M)]_n)\label{equation:step4}
	\end{align}
	for all $n\geq 0$. Since  $ \F_I(J) = R(J)/IR(J) \cong A/I[X_1,\ldots,X_d]$ and 
	$[R(\M)/(R(\M)_+(1)+R(\M_J))]_1=M_1/(M_2+JM)$, we have that 
	\begin{align*}
		\ell_A\Big(\Big[\Big[\frac{R(\M)}{R(\M)_+(1)+R(\M_J)}\Big]_1\otimes \F_I(J)\Big]_{n-1}\Big)&= \ell_A\Big(\frac{M_1}{M_2+JM}\Big)\binom{n-1+d-1}{d-1}\\
		&= \ell_A\Big(\frac{M_1}{M_2+JM}\Big)
		\Big(\binom{n+d-1}{d-1}-\binom{n+d-2}{d-2}\Big).		
	\end{align*}
    Thus,
	\begin{align*}
		\ell_A\Big(\frac{M_n}{M_{n+1}}\Big)&=\ell_A\Big(\frac{M}{M_1}\Big)\binom{n+d-1}{d-1}+\ell_A\Big(\frac{M_1}{M_2+JM}\Big)\binom{n+d-2}{d-1}-\ell_A([K^{(1)}(\M)]_n)\nonumber\\
		&\qquad + \ell_A(S_J(\M)_{n-1})-\ell_A([K^{(3)}(\M)]_{n-1})-\ell_A([K^{(2)}(\M)]_n)\\
		&= \ell_A\Big(\frac{M}{M_1}\Big)\binom{n+d-1}{d-1}+\ell_A\Big(\frac{M_1}{M_2+JM}\Big)\Big\{\binom{n+d-1}{d-1}-\binom{n+d-2}{d-2}\Big\}\\
		&\qquad + \ell_A(S_J(\M)_{n-1})-\ell_A([K^{(1)}(\M)]_n)-\ell_A([K^{(3)}(\M)]_{n-1})-\ell_A([K^{(2)}(\M)]_n)\\
		&=\Big\{ \ell_A\Big(\frac{M}{M_1}\Big)+ \ell_A\Big(\frac{M_1}{M_2+JM}\Big)\Big\}
        \binom{n+d-1}{d-1}-\ell_A\Big(\frac{M_1}{M_2+JM}\Big)\binom{n+d-2}{d-2}\\
		&\qquad +\ell_A(S_J(\M)_{n-1})-\ell_A([K^{(1)}(\M)]_n)-\ell_A([K^{(3)}(\M)]_{n-1})-\ell_A([K^{(2)}(\M)]_n)
	\end{align*}
	for all $n\geq 0$. On comparing the coefficients, we get that
	\begin{align*} 
		e_0(\M)&=\begin{cases}
			\ell_A\Big(\frac{M}{M_1}\Big)+ \ell_A\Big(\frac{M_1}{M_2+JM}\Big)+e_0(S_J(\M))-\mathop\sum\limits_{i\in\Gamma} e_0(K^{(i)}(\M)) & \text{  if } \dim S_J(\M)=d\\        
			\ell_A\Big(\frac{M}{M_1}\Big)+ \ell_A\Big(\frac{M_1}{M_2+JM}\Big)-\mathop\sum\limits_{i\in\Gamma} e_0(K^{(i)}(\M)) & \text{  if } \dim S_J(\M)\leq d-1
		\end{cases}\\
		&= 2\ell_A\Big(\frac{M}{M_1}\Big)+ \ell_A\Big(\frac{M_1}{M_2+JM}\Big)+e_1(\mathcal{M})-e_1^J(M)-e_0(\M)-\mathop\sum\limits_{i\in\Gamma}e_0(K^{(i)}(\M))
	\end{align*}
	since $e_1(\M)-e_0(\M)-e_1^J(M)+\ell(M/M_1)=e_0(S_J(\M))  \text{ if } \dim S_J(\M)=d$ and
	$e_1(\M)-e_0(\M)-e_1^J(M)+\ell(M/M_1)=0  \text{ if } \dim S_J(\M)\leq d-1$
	using Corollary \ref{corr-2.5}. 
\end{proof}

In view of Lemma \ref{lemma-2.2(a)}, it is easy to observe that $\ass_{R(J)} K^{(i)}(\M)\subseteq \{\m R(J)\}$ for $1\leq i\leq 3.$ To see this, note that 
$K^{(1)}(\M)$ is embedded in $[R(\M)/(R(\M)_+(1)+R(\M_J))]_1\otimes \F_I(J) \cong (M_1/(M_2+JM))\otimes (A/I)[X_1,\ldots,X_d]$ which is a
maximal Cohen-Macaulay $\F_I(J)\cong A/I[X_1,\ldots,X_d]$-module, 
$K^{(2)}(\M)$ is embedded in $ \F^\prime\cong M/M_1[X_1,\ldots,X_d]$
and  $K^{(3)}(\M)$ is embedded in $ S_J(\M)$. Therefore, if $K^{(i)}(\M)\neq 0$, then $\dim K^{(i)}(\M)=d$ for 
$1\leq i\leq 3.$ We will need the following proposition later. Here note that $(C_1)\Rightarrow (C_0)$ and $(C_1)+(C_3)\Rightarrow R(\M_J)$ satisfies $(S_2)$ see \cite[Theorem 6.2]{trung-gcm}. 
\begin{proposition}\label{proposition-3.5}
 Let $M$ be a finitely generated faithful flat $A$-module. Suppose $\M=\M_L$ and $J\subsetneq \m $ or $\M=\M_\I$ and $J\subsetneq I_1$. Suppose the conditions $(C_1)$ and $(C_2)$ are satisfied by $a_1,\ldots,a_d$ for both $A$ and $M$. Also, assume that $(C_3)$ is satisfied by $A$ and $M.$  Then the following conditions are equivalent:
	\begin{enumerate}[label=(\roman*)]
		\item\label{proposition-3.5-equivalent-condition-1}  We have
		$e_1(\mathcal{M})-e_1^J(M)=2e_0(\M)-2\ell_A(M/M_1)-\ell_A(M_1/(M_2+JM)).$
		\item\label{proposition-3.5-equivalent-condition-2} There exist  exact sequences
		$$0\to ((M_1/(M_2+JM))\otimes\F_I(J))(-1)\to R(\M)/(R(\M)_+(1)+R(\M_J))\to S_J(\M)(-1)\to 0,$$
		
		$$0\to \F_{I_1}(\M_J)
        \to G(\M) \to R(\M)/(R(\M)_+(1)+R(\M_J))\to 0$$
         and 
         $$0\to \F_{\m}(\M_J)
        \to G(\M) \to R(\M)/(R(\M)_+(1)+R(\M_J))\to 0$$
		of graded $R(J)$-modules.
		\item \label{proposition-3.5-equivalent-condition-3} $K^{(i)}(\M)=0$ for $1\leq i\leq 3.$
	\end{enumerate}
	When this is the case, we have the injective maps 
	$$\h_{\xi}^i(G(\M))\hookrightarrow \h_{\xi}^i(R(\M)/(R(\M)_+(1)+R(\M_J)))\hookrightarrow \h_{\xi}^i(S_J(\M))(-1)$$
	of graded $R(J)$-modules for $0\leq i\leq d-1$ and 
	$\h_\xi^i(G(\M))\cong \h_{\xi}^i(R(\M)/(R(\M)_+(1)+R(\M_J)))\cong \h_{\xi}^i(S_J(\M))(-1)$ for $0\leq i\leq d-2$.
\end{proposition}
\begin{proof} We show that \ref{proposition-3.5-equivalent-condition-1}$\Rightarrow$ \ref{proposition-3.5-equivalent-condition-3} $\Rightarrow$ \ref{proposition-3.5-equivalent-condition-2}$\Rightarrow$ \ref{proposition-3.5-equivalent-condition-1}.
	\item \ref{proposition-3.5-equivalent-condition-1}  $\Rightarrow$ \ref{proposition-3.5-equivalent-condition-3} 
	Suppose \ref{proposition-3.5-equivalent-condition-1} holds and $K^{(i)}(\M)\neq 0$ for some $1\leq i\leq 3$. Then $i\in\Gamma$ and $e_0(K^{(i)}(\M))\neq 0$ using Theorem \ref{lemma-3.3} which is  a contradiction. 
	
	\item \ref{proposition-3.5-equivalent-condition-3}  $\Rightarrow$ \ref{proposition-3.5-equivalent-condition-2}
	From the exact sequence  \eqref{equation:step3}, we get that $\im \phi\cong ((M_1/(M_2+JM))\otimes\F_I(J))(-1)$. The first exact sequence follows from the exact sequence in \eqref{equation:step2}. Second and third exact sequences follow from the exact sequences in \eqref{equation:step1} and \eqref{equation:step1(a)}, respectively. 
	\item \ref{proposition-3.5-equivalent-condition-2}  $\Rightarrow$ \ref{proposition-3.5-equivalent-condition-1} Suppose \ref{proposition-3.5-equivalent-condition-2}
	holds. Then, for all $n\geq 0$,
	\begin{align*}
		\ell_A(M_n/M_{n+1})=\ell_A(G(\mathcal{M})_n)&=\ell_A(M/M_1)\binom{n+d-1}{d-1}+\ell_A\Big(\Big[\frac{R(\M)}{R(\M)_{+}(1)+R(\M_J)}\Big]_n\Big)\\
		&=\ell_A(M/M_1)\binom{n+d-1}{d-1}+\ell_A\Big(\Big[\frac{M_1}{M_2+JM}\otimes\F_I(J)\Big]_{n-1}\Big)\\
        &\qquad+\ell_A((S_J(\M))_{n-1}).
	\end{align*}
	We get the desired equality using the same argument as done in Theorem \ref{lemma-3.3}.
	
	Finally, since $\F_I(J) \cong (A/I)[X_1,\ldots,X_d]$ by Lemma \ref{lemma-2.2(a)} and  $(M_1/(M_2+JM))\otimes\F_I(J)$ is a maximal 
	Cohen-Macaulay module, we get the injective maps 
	$$\h_\xi^i(G(\M))\hookrightarrow \h_\xi^i(R(\M)/(R(\M)_+(1)+R(\M_J)))\hookrightarrow \h_\xi^i(S_J(\M))(-1)$$
	of graded $R(J)$-modules for $0\leq i\leq d-1$ 
	and $\h_\xi^i(G(\M))\cong \h_{\xi}^i(R(\M)/(R(\M)_+(1)+R(\M_J)))\cong \h_{\xi}^i(S_J(\M))(-1)$ for $0\leq i\leq d-2$.
\end{proof}
\begin{corollary}\label{corollary-3.6} Let $M$ be a finitely generated faithful flat $A$-module and $\M=\M_L$ or $\M=\M_\I.$ Suppose the conditions $(C_1)$ and $(C_2)$  are satisfied by $a_1,\ldots,a_d$ for both $A$ and $M.$ Also, assume that $C_3$ is satisfied by $A$ and $M$ and  $e_1(\M)-e_1^J(M)=2e_0(\M)-2\ell(M/M_1)-\ell(M_1/(M_2+JM))$ holds. Then the following conditions hold:
	\begin{enumerate}[label=(\roman*)]
		\item $M_{n+2}\subseteq J^n M_1$ for all $n\geq 0, $
		\item $J^nM\cap M_{n+1}=J^n M_1$ for all $n\geq 0,$
		\item $\depth G(\mathcal{\M})>0$ and 
		\item $(a_1,\ldots,\check{a_i},\ldots,a_d)M:_Ma_i\subseteq M_2+JM$ for all $1\leq i\leq d.$
	\end{enumerate}
\end{corollary}
\begin{proof}
	If either $J=I_1$ or $J=\m$, then by Remark \ref{remark-I1=Q-case}, $M$ is Cohen-Macaulay $A$-module and $\M=\M_J$. Then the conclusion holds. 
	Suppose $\M=\M_L$ and $J\subsetneq \m $ or $\M=\M_\I$ and $J\subsetneq I_1.$ 
	\begin{enumerate}[label=(\roman*)]
		\item By Proposition \ref{proposition-3.5}, $K^{(3)}(\M)=0$ which gives $M_{n+2}\subseteq J^n M_1 $ for all $n\geq 0$. 
		\item By Proposition \ref{proposition-3.5}, $K^{(2)}(\M)=0$, so $[K^{(2)}(\M)]_n=(J^nM\cap M_{n+1})/J^n M_1=0$ for all $n\geq 0.$
		\item Using the two exact sequences of Proposition \ref{proposition-3.5}\ref{proposition-3.5-equivalent-condition-2}  and depth lemma, it is enough to see that $(M_1/(M_2+JM))\otimes \F_I(J)\cong (M_1/(M_2+JM))\otimes A/I[X_1,\ldots,X_d]$ and $S_J(\M)$ have positive depths as $R(J)$-modules. 
		\item Let $x\in (a_1,\ldots,\check{a_i},\ldots,a_d)M:_Ma_i$. Since $a_1,\ldots,a_d$ satisfies $(C_2)$ for $M$, hence $x\in IM$. Then $xa_i=\mathop\sum\limits_{1\leq j\leq d, j\neq i} a_jx_j$, $x_j\in A$. Here, $x,x_j\in IM\subseteq M_1$ due to the condition $(C_2)$. Let 
		$$g=\overline{x}\otimes \overline{a_it}-\mathop\sum\limits_{1\leq j\leq d, j\neq i} \overline{x_j}\otimes \overline{a_jt}\in (M_1/(M_2+JM))\otimes \F_I(J)$$
		where $\overline{x}, \overline{x_j}$ denote the images of $x,x_j$ in $M_1/(M_2+JM)$ and $\overline{a_it}$ denote the image of $a_it$ in $\F_I(J)$. The image of $g$ under the map $0\to ((M_1/(M_2+JM))\otimes \F_I(J))(-1)\to R(\M)/(R(\M)_+(1)+R(\M_J))$ is 
		$$\overline{xa_it^2}-\mathop\sum\limits_{1\leq j\leq d,j\neq i} \overline{x_ja_jt^2}=0$$
		which implies that $g=0\in (M_1/(M_2+JM))\otimes \F_I(J) \cong (M_1/(M_2+JM))\otimes (A/I)[X_1,\ldots,X_d]$. Therefore, $x\in M_2+JM.$\qedhere
	\end{enumerate}
\end{proof}
\section{\texorpdfstring
{A necessary and sufficient condition for equality in \eqref{very-1st-eqn(a)}}
{A necessary and sufficient condition for equality in the inequality}
}\label{Section-proof-1}
In this section, we prove a necessary and sufficient condition for the equality
$e_1(\M)-e_1^J(M)=2e_0(\M)-2\ell_A(M/M_1)-\ell_A(M_1/(M_2+JM))$ to hold. 
We begin with the following remark. In what follows, we use the inequality $a(G(\M))+d\leq r_J(\M)$,
whose proof is similar to that of \cite[Proposition 3.2]{nvt} and so we skip providing it here. Recall that 
 $$a(G(\M)):=\max\{n\in \mathbb{Z}~|~ [\h_\xi^d(G(\M))]_n\neq (0)\}.$$

\begin{remark}\label{remark-5.1}Let $M$ be a finitely generated faithful flat $A$-module and $\M= \M_L$ or $\M=\M_\I.$
	Suppose $M_2\subseteq JM$ and  the equality $e_1(\M)-e_1^J(M)=2e_0(\M)-2\ell(M/M_1)-\ell(M_1/(M_2+JM))$ holds. Then
	$M$ is Cohen-Macaulay $A$-module with 
	$M_{n+1}=J^nM_1$ for $n\geq 1$,  $S_J(\M)=(0)$ and $G(\mathcal{M})$ is Cohen-Macaulay $A$-module with $a(G(\M))\leq 1-d.$ 
\end{remark}
\begin{proof}
	By Proposition \ref{proposition:p1},  $e_1(\M)-e_1^J(M)\geq e_0(\M)-\ell(M/M_1)$. Hence,
	\begin{eqnarray*}
		0&=& e_1(\M)-e_1^J(M)-2e_0(\M)+2\ell(M/M_1)+\ell(M_1/JM)\\
		&=&e_1(\M)-e_1^J(M)-2e_0(\M)+\ell(M/M_1)+\ell(M/JM)\\
		&\geq& \ell(M/JM)-e_0(\M)\\
		&\geq& 0.
	\end{eqnarray*}
	Thus, $\ell_A(M/JM)=e_0^J(M)$ which implies that $M$ is Cohen-Macaulay $A$-module, see \cite[Corollary 4.7.11]{bruns-herzog} and $e_1(\M)-e_1^J(M)=e_0(\M)-\ell(M/M_1)$. By Lemma \ref{lemma-2.3(a)} and Lemma \ref{lemma-2.3(b)}, either $S_J(\M)=(0)$ or $\dim S_J(\M)=d$. Using Corollary \ref{corr-2.5}, we get that $S_J(\M)=(0)$.
	 Consequently, $M_{n+1}=J^nM_1$ for
	all $n\geq 1$ and 
	$G(\M)$ is Cohen-Macaulay $A$-module by Vallabrega-Valla criteria. Since $G(\M)$ is Cohen-Macaulay $A$-module, we have $a(G(\M))+d\leq 1$. 
\end{proof}

\begin{proposition}\label{theorem-4.3}
 Let $M$ be a finitely generated faithful flat $A$-module and $\M=\M_L$ or $\M=\M_\I$. Suppose $(C_1)$ and $(C_2)$ conditions are satisfied by $a_1,\ldots,a_d$ for both $A$ and $M$.  Let the equality
	$e_1(\M)-e_1^J(M)=2e_0(\M)-2\ell_A(M/M_1)-\ell_A(M_1/(M_2+JM))$ hold. Then $M_{n+2}\subseteq J^nM_2+W$ for $n\geq 1$ and $W\subseteq M_2$.
\end{proposition}
\begin{proof}
 We first show that $W\subseteq M_2$.  By Lemma \ref{lemma-reduction-steps}(iii), we have that $W\subseteq M_2+JM$ and the equality \eqref{very-1st-eqn(a)} holds in $\overline{M}$. Then by Corollary \ref{corollary-3.6}, $J^n\overline{M}\cap \overline{M_{n+1}}=J^n\overline{M_1}$ for $n\geq 0$. 
 Now let $w=i+q\in W$ where $i\in M_2$ and $q\in JM$. Then $q=w-i\in (JM+W)\cap (M_2+W)=JM_1+W$.
	So, $q\in (JM_1+W)\cap JM=JM_1$ since $JM\cap W=0$ which implies $w=i+q\in M_2$. Hence $W\subseteq M_2$.
	
	For the rest of the proof, we may pass to the ring $\overline{A}$, module $\overline{M}$ and assume that the conditions $(C_1)$ and $(C_2)$ are satisfied by the images of $a_1,\ldots,a_d$ in $\overline{A}$ for both $\overline{A}$ and $\overline{M}.$ Moreover,  $\overline{A}$ and $\overline{M}$ both satisfy $(C_3)$. 
	We show, by induction on $d$, that $M_{n+2}= J^n M_2$ for all $n\geq 1$. Suppose the equality  $e_1(\M)-e_1^J(M)=2e_0(\M)-2\ell_A(M/M_1)-\ell_A(M_1/(M_2+JM))$
	holds. Then
	by Corollary \ref{corollary-3.6}, $\depth G(\M)>0$ and $M_{n+2}\subseteq J^n M_1$ for all $n\geq 0.$
	We may assume that 
	$f_1=a_1t\in R(J)$ is a regular element of $G(\M)$.
	
	Suppose $d=1$. Since $M_{n+2}\subseteq J^n M_1\subseteq a_1^nt^nM$ and $f_1^n$ is a regular element of $G(\M)$,  we get that $M_{n+2}=M_{n+2}\cap a_1^nt^nM=a_1^nt^nM_2= J^nM_2$ for all $n\geq 1$.
	
	Suppose $d\geq 2.$ As earlier, let $M^\prime=M/a_1M,$ 
	$\M^\prime=\{M_n^\prime=(M_n+a_1M)/a_1M\}_{n\geq0}$ and $\M_J^\prime=\{(J^nM+a_1M)/a_1M\}_{n\geq 0}$ be the filtration of $M^\prime$. Then the conditions  $(C_1)$ and $(C_2)$ are satisfied for $M^\prime$ and
	$e_1(\M^\prime)-e_1^J(M^\prime)=2e_0(\M^\prime)-2\ell_A(M^\prime/{M_1}^\prime)-\ell_A({M_1}^\prime/({M_2}^\prime+JM^\prime))$ by Lemma \ref{lemma-reduction-steps}.
	Further, $(C_1)$, $(C_2)$ and $(C_3)$ are satisfied for $\overline{M^\prime}=M^\prime/\h_{\m}^0(M^\prime)$. 
	Let $\overline{\M^\prime} =\{(M_n^\prime+\h_{\m}^0(M^\prime))/\h_{\m}^0(M^\prime)\}$ be the $I^\prime \overline{A^\prime}$-good filtration of $\overline{M^\prime}$, where $\overline{A^\prime}=A^\prime/\h_{\m}^0(A^\prime)$. 
	By induction hypothesis, we have for all $n\geq 1$,
	\begin{eqnarray*}
		&M_{n+2}^\prime&\subseteq  J^n M_2 ^\prime+\h_{\m}^0(M^\prime)\\
		\Rightarrow & M_{n+2} & \subseteq J^n M_2+a_1M+\h_{\m}^0(M^\prime)\\
		&& \subseteq J^n M_2+(a_1M:J)
	\end{eqnarray*}
	since $M_{n+2}\subseteq J^nM_1$ by Corollary \ref{corollary-3.6}, $a_1,\ldots,a_d$ is a d-sequence of $M$ and $\h_{\m}^0(M^\prime)=(a_1M:J)$. Now we show that $M_{n+2}\subseteq J^nM_2$ by induction on $n$. let $x\in M_{n+2}$, write $x=y+z$ where $y\in J^nM_2$ and 
	$z\in (a_1M:J)$. Then $z=x-y\in (a_1M:J)\cap M_{n+2}\subseteq (a_1M:J)\cap J^n M_1\subseteq (a_1M:J)\cap JM =a_1M$ since
	$M_{n+2}\subseteq J^nM_1$ by Corollary \ref{corollary-3.6} and $a_1,\ldots,a_d$ is a d-sequence of $M$.  This implies $$z\in a_1M\cap M_{n+2}=a_1M_{n+1}\subseteq JM_{n+1}$$ as $a_1t$ is a regular element for $G(\mathcal{M})$. 
	For $n=1$, we get that $x=y+z\in JM_2$. Suppose $n>1$. Then $z\in JM_{n+1}\subseteq J\cdot J^{n-1}M_2$ by induction hypothesis which implies $x\in J^nM_2$. This completes the proof.
\end{proof}
Now we prove the main result of this section. 
	\begin{theorem}\label{the-theorem}
     Let $M$ be a finitely generated faithful flat $A$-module and $\M=\M_L$ or $\M=\M_\I$. Suppose $(C_1)$ and $(C_2)$ conditions are satisfied by $a_1,\ldots,a_d$ for both $A$ and $M$. Then the following conditions are equivalent: 
	\begin{enumerate}
		\item \label{the-theorem-equivalent-cond-1} $e_1(\M)-e_1^J(M)=2e_0(\M)-2\ell(M/M_1)-\ell(M_1/(M_2+JM))$;
		\item \label{the-theorem-equivalent-cond-2} $M_{n+2}\subseteq J^nM_2+W,~(J^nM+W)\cap (M_{n+1}+W)=J^nM_1+W$ for all $n\geq 1$ and $(a_1,\ldots,\check{a_i},\ldots,a_d)M:_Ma_i\subseteq M_2+JM$ for $1\leq i\leq d$.
	\end{enumerate}
	When this is the case, we have $W\subseteq M_2$. 
\end{theorem}
\begin{proof}
 $\eqref{the-theorem-equivalent-cond-1}\Rightarrow \eqref{the-theorem-equivalent-cond-2}$ By Proposition \ref{theorem-4.3},  $M_{n+2}\subseteq J^n M_2+W$ for $n\geq 1$ and $W\subseteq M_2$. By Lemma \ref{lemma-reduction-steps}, we have $e_1(\overline{\M})-e_1^J(\overline{M})=2e_0(\overline{\M})-2\ell_A(\overline{M}/\overline{M_1} )-\ell_A(\overline{M_1} /(\overline{M_2} +J\overline{M}))$
and $W\subseteq M_2+JM$. Since 
$(C_1)$, $(C_2)$ and $(C_3)$ are satisfied for $\overline{M}$, by Corollary \ref{corollary-3.6} we get $(J^nM+W)\cap (M_{n+1}+W)=J^nM_1+W$ for $n\geq 1$ and 
$(a_1,\ldots,\check{a_i},\ldots,a_d)M:_M a_i\subseteq M_2+JM$ for all $1\leq i\leq d;$. 

 $\eqref{the-theorem-equivalent-cond-2}\Rightarrow \eqref{the-theorem-equivalent-cond-1}$ 	Suppose $J=I_1$ or $J=\m$. Hence, $JM=M_1$ for both filtrations. Then $M_2+W=(M_2+W)\cap (JM+W)=JM_1+W=J^2M+W$ and $M_{n+2}\subseteq J^{n+2}M+W$ for $n\geq 1.$ This implies that $\overline{\M}=\overline{\M_J}$ and $e_1(\overline{\M})-e_1^J(\overline{M})=0$. Since $W\subseteq (a_1,\ldots,\check{a_i},\ldots,a_d)M:_Ma_i\subseteq M_2+JM$,  it is enough to get the equality $$e_1(\overline{\M} )-e_1^J(\overline{M})=0=2(e_0(\overline{\M})-\ell(\overline{M}/J\overline{M}))$$
	in $\overline{M}$ and apply Lemma \ref{lemma-reduction-steps}. By \cite[Theorem 4.1]{trung-abs}, we have that $e_0(\overline{\M})=e_0^J(\overline{M})=\ell(\overline{M}/J\overline{M})-\ell(((a_1,\ldots,a_{d-1})M:_Ma_d)/ (((a_1,\dots,a_{d-1})M:_Ma_d)\cap JM))=\ell(\overline{M}/J\overline{M})$ as $(a_1,\ldots,a_{d-1})M:_Ma_d\subseteq M_2+JM=JM.$
	
	Now let $JM\varsubsetneq M_1.$ Hence, $J\varsubsetneq I_1.$
	Since $a_1,\ldots,a_d$ satisfy $(C_1)$ for both $M$ and $A$, by Remark \ref{remark-nota}, 
	$W\subseteq (a_1,\ldots,\check{a_i},\ldots,a_d)M:_Ma_i\subseteq M_2+JM$. We show the required equality in the module $\overline{M}$ and use Lemma \ref{lemma-reduction-steps}. Passing to the ring $\overline{A}$ and module $\overline{M}$, we assume that the conditions $(C_1)$ and $(C_2)$ are satisfied by the images of $a_1,\ldots,a_d$ in $\overline{A}$ for both $\overline{A}$ and $\overline{M}.$ Moreover, since $\overline{A}$ and $\overline{M}$ both satisfy $(C_3)$,
	$M_{n+1}=J^{n-1}M_2$ for $n\geq 2$  and $J^nM\cap M_{n+1}=J^nM_1$ for all $n\geq 1$. By Proposition \ref{proposition-3.5}, it is enough to show that $K^{(i)}(\M)=0$ for $1\leq i\leq 3$. 
	It follows easily that 
	$$K^{(3)}(\M)=\mathop\oplus\limits_{n\geq 1}\frac{M_{n+2}+J^nM_1 }{J^nM_1 }=0 \text{ and }
	K^{(2)}(\M)=\mathop\oplus\limits_{n\geq 1}\frac{J^nM\cap M_{n+1}}{ J^nM_1}=0.$$
	Now we show that $K^{(1)}(\M)=0.$ From \eqref{equation:step3},  we have the exact sequence 
	\begin{equation*}
		0\to K^{(1)}(\M)\to ([R(\M)/(R(\M)_+(1)+R(\M_J))]_1\otimes \F_I(J))(-1)\xrightarrow{\psi_3}\im \phi\to 0
	\end{equation*}
	where $K^{(1)}(\M)=\ker\psi_3$. Suppose $K^{(1)}(\M)\neq 0$. Let 
	$n=\min\{s\mid [K^{(1)}(\M)]_s\neq 0 \}.$
	Since $\big[[R(\M)/(R(\M)_+(1)+R(\M_J))]_1\otimes \F_I(J)\big]_{s-1}=0$ for all $s\leq 0$, we have $[K^{(1)}(\M)]_s=0$ for $s\leq 0.$
	For $s=1$, notice that 
	$$\big[[R(\M)/(R(\M)_+(1)+R(\M_J))]_1\otimes \F_I(J)\big]_{0}\cong (M_1/(M_2+JM))\otimes A/I\cong M_1/(M_2+JM).$$ 
	
	Let $g=\mathop\sum\limits_{1\leq i\leq p} \overline{z_i}\otimes\overline{r_i} \in [K^{(1)}(\M)]_1$ where $\overline{z_i}$ denote the image of 
	$z_i\in M_1$ in $M_1/(M_2+JM)$ and $\overline{r_i}$ denote the image of $r_i$ in $A/I$.  
	Then $\psi_3(g)=\mathop\sum\limits_{1\leq i\leq p} \overline{r_itz_i}=0$
	where $\overline{r_itz_i}$ denote the image of $r_itz_i$ in $\im \phi \subseteq [R(\M)/(R(\M)_+(1)+R(\M_J))]_1=M_1/(M_2+JM)$ which implies $$\mathop\sum\limits_{1\leq i\leq p} r_iz_i\in M_2+JM \Rightarrow
	g=\mathop\sum\limits_{1\leq i\leq p} \overline{z_i}\otimes \overline{r_i}=\mathop\sum\limits_{1\leq i\leq p}\overline{r_iz_i}\otimes \bar{1}=0$$
    in $\big[[R(\M)/(R(\M)_+(1)+R(\M_J))]_1\otimes \F_I(J)\big]_{0}$. 
	This gives that $[K^{(1)}(\M)]_1=0$. 
	Therefore, $n\geq 2.$ Let $0\neq g\in [K^{(1)}(\M)]_n=[\ker \psi_3]_n.$ Put 
	\begin{align*}
		\Lambda=\{(\alpha_1,\ldots,\alpha_d)\in \mathbb{Z}^d~|~  \alpha_i\geq 0 \text{ for } 1\leq i\leq d \text{ and } \mathop\sum_{i=1}^d \alpha_i=n-1\}.
	\end{align*}
	We may write 
	\begin{eqnarray*}
		g&=&\mathop\sum\limits_{\alpha \in\Lambda, \alpha_1\geq 1} \overline{x_\alpha}\otimes \overline{(a_1t)^{\alpha_1}(a_2t)^{\alpha_2}\ldots(a_dt)^{\alpha_d}}
		+ \mathop\sum\limits_{\beta\in\Lambda,\beta_1=0} \overline{x_\beta}\otimes \overline{(a_2t)^{\beta_2}\ldots(a_dt)^{\beta_d}}
	\end{eqnarray*}
	where $\overline{x_{\alpha}}$ and $\overline{x_{\beta}}$ denote the images of $x_{\alpha},x_{\beta}\in M_1$ in $M_1/(M_2+JM)$, and
	$\overline{(a_1t)^{\alpha_1}(a_2t)^{\alpha_2}\ldots(a_dt)^{\alpha_d}}$ and $\overline{(a_2t)^{\beta_2}\ldots(a_dt)^{\beta_d}}$
	denote the images of corresponding elements in $\F_I(J)_{n-1}$. Then
	\begin{align*}
		\psi_3(g)&=\mathop\sum\limits_{\alpha \in\Lambda, \alpha_1\geq 1} \overline{a_1^{\alpha_1}a_2^{\alpha_2}\ldots a_d^{\alpha_d}t^nx_{\alpha}}+\mathop\sum\limits_{\beta\in\Lambda,\beta_1=0}\overline{ a_2^{\beta_2}\ldots a_d^{\beta_d}t^nx_\beta}=0\in \Big[\frac{R(\M)}{R(\M)_+(1)+R(\M_J)}\Big]_n\\
		\Rightarrow &\qquad\mathop\sum\limits_{\alpha \in\Lambda, \alpha_1\geq 1} a_1^{\alpha_1}a_2^{\alpha_2}\ldots a_d^{\alpha_d}x_{\alpha}+
		\mathop\sum\limits_{\beta\in\Lambda,\beta_1=0} a_2^{\beta_2}\ldots a_d^{\beta_d}x_\beta \in M_{n+1}+J^nM.
	\end{align*}
	
	We have for all $n\geq 2$, $M_{n+1}+J^nM=J^{n-1}M_2+J^nM\subseteq a_1J^{n-2}M_2+(a_2,\ldots,a_d)^{n-1}M_2+a_1J^{n-1}M+(a_2,\ldots,a_d)^nM\subseteq a_1J^{n-2}M_2+a_1J^{n-1}M+(a_2,\ldots,a_d)^{n-1}M$
	and $\mathop\sum\limits_{\beta\in\Lambda,\beta_1=0}a_2^{\beta_2}\ldots a_d^{\beta_d}x_\beta \in (a_2,\ldots,a_d)^{n-1}M$. So, there exists 
	$q\in J^{n-2}M_2$ and $q^\prime \in J^{n-1}M$ such that 
	\begin{eqnarray*}
		&\mathop\sum\limits_{\alpha \in\Lambda, \alpha_1\geq 1} a_1^{\alpha_1}a_2^{\alpha_2}\ldots a_d^{\alpha_d}x_{\alpha}+a_1q+a_1q^\prime\in(a_2,\ldots,a_d)^{n-1}M\\
		\Rightarrow& a_1\cdot\{\mathop\sum\limits_{\alpha \in\Lambda, \alpha_1\geq 1} a_1^{\alpha_1-1}a_2^{\alpha_2}\ldots a_d^{\alpha_d}x_{\alpha}+q+q^\prime\} \in(a_2,\ldots,a_d)^{n-1}M\\
		\Rightarrow& \mathop\sum\limits_{\alpha \in\Lambda, \alpha_1\geq 1} a_1^{\alpha_1-1}a_2^{\alpha_2}\ldots a_d^{\alpha_d}x_{\alpha}+q+q^\prime\in(a_2,\ldots,a_d)^{n-2}(M_2+JM)\subseteq M_n+J^{n-1}M
	\end{eqnarray*}
	since $(a_2,\ldots,a_d)^{n-1}M:_Ma_1\subseteq (a_2,\ldots,a_d)^{n-2}.((a_2,\ldots,a_d)M:_Ma_1)$ by \cite[Lemma 3.5]{goto-ozeki-ncm-2010} and 
	$(a_2,\ldots,a_d)M:_Ma_1\subseteq M_2+JM$ is given. Hence, we get
	\begin{eqnarray*}\mathop\sum\limits_{\alpha \in\Lambda, \alpha_1\geq 1} a_1^{\alpha_1-1}a_2^{\alpha_2}\ldots a_d^{\alpha_d}x_{\alpha}&\in& M_n+J^{n-1}M\\
		\Rightarrow
	\mathop\sum\limits_{\alpha\in\Gamma_1} \overline{x_\alpha}\otimes \overline{(a_1t)^{\alpha_1-1}(a_2t)^{\alpha_2}\ldots(a_dt)^{\alpha_d}}&\in& [K^{(1)}(\M)]_{n-1}=0\end{eqnarray*}
	which gives that  
	$$g=\mathop\sum\limits_{\beta\in\Gamma_2}\overline{x_\beta}\otimes \overline{(a_2t)^{\beta_2}\ldots(a_dt)^{\beta_d}}.$$ 
By symmetry of the elements $a_it$, the similar steps can be repeated for any $a_i$, $1\leq i\leq d.$
Therefore, $g=0$ which is a contradiction. Hence $K^{(1)}(\M)=0.$
	This completes the proof.
\end{proof}

\section{A detour to generalized depth and generalized Cohen-Macaulayness}\label{Section-detour}
At this point, we take a detour to gather facts on the finiteness properties of the local cohomology modules. The results of this section will be used in proving our main theorems in Section \ref{Section-main-1}. We first recall the notion of generalized depth, defined in \cite{HM94}. We refer to  the papers of
	Brodmann \cite{Brodmann} and Faltings \cite{Faltings} for this topic.  
	For an $A$-module $M$  and an ideal $I$ of $A$, the generalized depth of $M$ with respect to $I$ is 
	$$g\text{-} \depth_I(M):=\sup\{k\in\mathbb{Z}\mid I\subseteq \sqrt{\ann_A\h_{\m}^i(M)} \text{ for all }i<k\}.$$
	Note that, $g\text{-}\depth_I M\geq k$ if and only if some power of $I$ annihilates $\h_\m^i(M)$ for $0\leq i\leq k-1.$
	For a  non-negatively graded Noetherian ring  $S=\mathop\oplus\limits_{n\geq 0} S_n$  with $S_0$ local and a homogeneous ideal $\mathcal{J}$ of $S$, we define the generalized depth of a graded $S$-module $P$ with respect to 
	$\mathcal{J}$ as, $$g\text{-}\depth_{\mathcal{J}} P=g\text{-}\depth_{\mathcal{J}L_\mathcal{E}} P_\mathcal{E}$$
	where $\mathcal{E}$ is the unique maximal homogeneous ideal of $S$. Investigating the relationship between the depths of $R(I)$ and $G(I)$, authors in  \cite[Proposition 3.2]{HM94} proved that 
	$$g\text{-}\depth_{R(I)_+} R(I)= g\text{-}\depth_{R(I)_+} G(I)+1.$$
    \begin{remark}
        We note that $$g\text{-} \depth_{G(I)_+} G(\I)\cong g\text{-} \depth_{G(\I)_+} G(\I). $$
        Since $R(I)$ inject in $R(\I),$ we have $H^i_{R(I)_+}(G(\I))\cong H^i_{R(I)_+R(\I)}(G(\I))\cong H^i_{R(\I)_+}(G(\I)).$ Last isomorphism happens because $\sqrt{R(I)_+R(\I)}=\sqrt{R(\I)_+}.$
    \end{remark}
	We extend the above result for $I$-good filtrations $\M=\M_L$ and $\M=\M_\I$ in Proposition \ref{Appendix-proposition-main-1}. First, we note the following lemma.
	
	
	\begin{lemma}\label{Appendix-Lemma-1}
		Let $M$ be a finitely generated $A$-module and $\M=\M_L$ or $\M=\M_\I.$ 
		Let $\p$ be a prime ideal of $R(I)$ such that $R(I)_+(1)\subseteq \p$ and $R(I)_+\nsubseteq \p$. Let $\q=\p/(R(I)_+(1)).$ Then 
		$$\depth R(\M)_{\p}=\depth G(\mathcal{M})_{\q}+1.$$ 
	\end{lemma}
	\begin{proof}
		Choose $xt^{m}\in R(I)_+\setminus \p$ with $x\in I^{m}$. Then  $(xt^{m-1})[R(\M)_{\p}]_n\subseteq (R(I)_+(1)) R(\M)_p\subseteq (R(\M)_+(1)_{\p}).$ 
  Indeed, $(xt^{m-1})R(\M)_{\p}=R(\M)_+(1)_{\p}.$ To see this, let $y=at^n\in R(\M)_+(1)$ with $a\in M_{n+1}$, then 
		$at^n=xt^{m-1}.\frac{at^{n+1}}{xt^m}\in xt^{m-1}R(\M)_{\p}$. Therefore  $R(\M)_+(1)_{\p}\cong xt^{m-1}R(\M)_{\p}$. Hence, $G(\M)_{\q}\cong R(\M)_{\p}/R(\M)_+(1)_{\p}\cong R(\M)_{\p}/(xt^{m-1})R(\M)_{\p}.$
		If $xt^{m-1}$ is a zero-divisor in  $R(\M)_{\p}$, then there exists an associated prime ${\p}^\prime$ of $R(\M)$ such that $xt^{m-1}\in {\p}^\prime\subseteq \p$.  Suppose ${\p}^\prime=(0:_{R(\M)} f)$, then $xt^{m-1}f=0\Rightarrow xt^{m}f=0\Rightarrow xt^m\in {\p}^\prime\subseteq 
		\p$, a contradiction. Thus, $xt^{m-1}$ is non-zero-divisor in 
		$R(\M)_{\p}$ which gives the desired result.
	\end{proof}
     
    The proof of the remark below is similar to that of \cite[Remark 2.2]{HM94}, but we include the steps for completeness.
	\begin{remark}\label{remark_6.3}
	    Let $(A,\m)$ be a complete local ring and and $p\subseteq q$ primes of $R(I).$ Let $M$ be a finitely generated $A$-module, then 
        $$
        \depth R(\M)_q+\dim R(I)/q\leq \depth R(\M)_p+\dim R(I)/p.
        $$
        \begin{proof}
            Since $(A,\m)$ is a complete local ring, $R(I)$ is a catenary ring and  $\dim R(I)/\p-\dim R(I)/\q= \rm{ht}~\q/\p.$ Thus, by localizing at $\q$, we may assume that $\q$ is the maximal ideal of $R(I).$ Since $R(\M)$ is a finitely generated $R(I)$-module, by Ischebeck's Theorem \cite[Theorem 17.1]{Mat}, $\Ext^i(R(I)/\p,R(\M))=0$ for $i\leq \depth R(\M)- \dim R(I)/\p.$ Therefore $\depth_\p R(\M)+ \dim R(I)/\p\geq \depth R(\M).$
        \end{proof}
	\end{remark}
	\begin{proposition}\label{Appendix-proposition-main-1}
		Let $M$ be a finitely generated $A$-module and $\M=\M_L$ or $\M=\M_\I.$ Then 
		$$g\text{-} \depth_{G(I)_+} G(\M)=g\text{-}\depth_{R(I)_+}R(\M)-1.$$
	\end{proposition}
	\begin{proof}
		We may assume that $A$ is complete, see \cite[Remark 2.4]{HM94}. Consequently, $A$ is a homomorphic image of a regular local ring and hence
		$R(I)$ and $G(I)$ are homomorphic images of regular rings. 
		Now using \cite[Proposition 2.4]{marley2}, we have that	
		\begin{align*}
			g\text{-} \depth_{G(I)_+} G(\mathcal{M})&=
			\min\{\depth G(\mathcal{M})_\q+\dim G(I)/\q \mid \q\in \rm{Spec}~ G(I), G(I)_+\nsubseteq \q,  \\& \qquad \q\text{ is homogeneous}\}\\
			g\text{-} \depth_{R(I)_+} R(\M)&=\min\{\depth R(\M)_\p+\dim R(I)/\p  \mid \p\in \rm{Spec}~ R(I), R(I)_+\nsubseteq \p,
		  \\&\qquad \p \text{ is homogeneous}\}. 
		\end{align*}
		
		Suppose $g\text{-} \depth_{R(I)_+}R(\M)\geq k$.
       Let $\q$ be a homogeneous prime ideal of $G(I)$ not containing $G(I)_+$. 
       Then there exists a prime ideal 
		 $\p$ in $R(I)$ such that $R(I)_+(1)\subseteq \p$ and $\q=\p/R(I)_+(1)$. Clearly, $R(I)_+\nsubseteq \p$ as $G(I)_+\nsubseteq \q$. 
		By Lemma  \ref{Appendix-Lemma-1}, $\depth R(\M)_\p=\depth G(\M)_\q+1$. Therefore,
		\begin{equation*}\depth G(\mathcal{M})_\q+\dim G(I)/\q =\depth R(\M)_\p-1+\dim R(I)/\p\geq k-1\end{equation*}
		which implies that $g\text{-} \depth_{G(I)_+} G(\mathcal{M})\geq g\text{-} \depth_{R(I)_+}R(\M) -1.$
		
		For the converse, let $g\text{-} \depth_{G(I)_+} G(\mathcal{M})\geq k$. Let $\p$ be a homogeneous prime ideal of $R(I)$ such that $R(I)_+\nsubseteq \p$. 
		Then, there exists a ${\p}^{\prime}\in \rm{Spec} ~R(I)$, $R(I)_+\nsubseteq \p^\prime$ and $(R(I)_+(1),\p)\subseteq \p^\prime.$ Otherwise, for each prime ideal $\p^\prime$ with 
		$(R(I)_+(1),\p)\subseteq \p^\prime$, we have $R(I)_+\subseteq \p^\prime$. This means $R(I)_+\subseteq \sqrt{(R(I)_+(1),\p)}$ which implies $R(I)_{+}^l\subseteq (R(I)_+(1),\p)$ for all $l\gg 0$. This gives $I^l=I^{l+1}+\p_l$, for $l\gg 0$. Then, by Nakayama lemma, $I^l\subseteq\p_l$ for $l\gg 0$. Thus $R(I)_+^l\subseteq \p\Rightarrow  R(I)_+\subseteq \p$, which is a contradiction. Therefore,  there exists $\p^\prime \in \Spec~ R(I)$ such that $(R(I)_+(1),\p)\subseteq\p^\prime$ and $R(I)_+\nsubseteq \p^\prime$. Now let $\q=\p^\prime/R(I)_+(1).$

		
        Then, by Lemma  \ref{Appendix-Lemma-1} and Remark \ref{remark_6.3}, we have
		\begin{align*}
			\depth R(\M)_\p+\dim R(I)/\p &\geq \depth R(\M)_{\p^\prime}+\dim R(I)/\p^\prime\\
			&\geq \depth G(\M)_\q+1+\dim G(I)/\q\\
			&\geq k+1.
		\end{align*}
		Hence, $g\text{-} \depth_{R(I)_+} R(\M)\geq k+1$ which gives 
		$g\text{-} \depth_{R(I)_+}R(\M)\geq g\text{-} \depth_{G(I)_+} G(\M)+1.$
	\end{proof}
	Next, we establish a relationship between the generalized depths of Sally module and associated graded module of $I$-good filtration $\M=\M_L$ or $\M=\M_\I$ in Proposition \ref{proposition-2.7}. The $I$-adic case was discussed in \cite{ozeki13}.
	First, we need the next two lemmas. 
	\begin{lemma}\label{lemma-2.6(a)}
     Let $M$ be a finitely generated faithful flat $A$-module and $\M=\M_\I.$ Suppose $(C_1)$ and $(C_2)$ conditions are satisfied by $a_1,\ldots,a_d$ for both $A$ and $M$. Then, for each $0\leq i\leq d,$ we have 
		$$[\h_\xi^i(I_1R(\M_J))]_n\cong\begin{cases} \h_\m^0(M)  & \mbox{ if } i=n=0\\
			\h_\m^{i-1}(M) & \mbox{ if } 3\leq i\leq d\mbox{ and } 2-i\leq n\leq -1\\
			(0) & \mbox{ otherwise }  
		\end{cases}$$
		for all $n\in\mathbb{Z}$. Hence $I_1 R(\M_J)$ is a generalized Cohen-Macaulay $R(J)$-module and $\dim_{R(J)} I_1 R(\M_J)=d+1$. 
	\end{lemma}
	\begin{proof}
		By Lemma \ref{lemma-2.2(a)}, $\F_{I_1}(\M_J)\cong (M/I_1M)[X_1,\ldots,X_d]$ is Cohen-Macaulay $R(J)$-module. So, by the exact sequence, $$0\to I_1 R(\M_J)\to R(\M_J)\to \F_{I_1}(\M_J)\to 0,$$ we get 
		$$\h_\xi^i(I_1 R(\M_J))\cong \h_\xi^i(R(\M_J))$$  for $0\leq i\leq d-1$ and the exact sequence $$0\to\h_\xi^d(I_1 R(\M_J))\to\h_\xi^d(R(\M_J))\to\h_\xi^d(\F_{I_1}(\M_J))$$
		of graded $R(J)$-modules.  Since $[\h_\xi^d(R(\M_J))]_n=0$ for all $n\leq 1-d$ by \cite[Theorem 6.2]{trung-gcm} and $[\h_\xi^d(\F_{I_1}(\M_J))]_n=0$ for all $n\geq 1-d$, we have 
		$\h_\xi^d(I_1 R(\M_J))\cong \h_\xi^d(R(\M_J))$ as graded $R(J)$-module. Now the conclusion follows by \cite[Theorem 6.2]{trung-gcm}.
	\end{proof}
    \begin{lemma}\label{lemma-2.6(b)}
   Let $M$ be a finitely generated faithful flat $A$-module and $\M=\M_L.$  Suppose $(C_1)$ and $(C_2)$ conditions are satisfied by $a_1,\ldots,a_d$ for both $A$ and $M$. Then, for each $0\leq i\leq d,$ we have 
		$$[\h_\xi^i(\m R(\M_J))]_n\cong\begin{cases} \h_\m^0(M)  & \mbox{ if } i=n=0\\
			\h_\m^{i-1}(M) & \mbox{ if } 3\leq i\leq d\mbox{ and } 2-i\leq n\leq -1\\
			(0) & \mbox{ otherwise }  
		\end{cases}$$
		for all $n\in\mathbb{Z}$. Hence $\m R(\M_J)$ is a generalized Cohen-Macaulay $R(J)$-module and $\dim_{R(J)} \m R(\M_J)=d+1$.
	\end{lemma}
	\begin{proof}
		By Lemma \ref{lemma-2.2(a)}, $\F_{\m}(\M_J)\cong (M/\m M)[X_1,\ldots,X_d]$ is Cohen-Macaulay $R(J)$-module. So, by the exact sequence, $$0\to \m R(\M_J)\to R(\M_J)\to \F_{\m}(\M_J)\to 0,$$ we get 
		$$\h_\xi^i(\m R(\M_J))\cong \h_\xi^i(R(\M_J))$$  for $0\leq i\leq d-1$ and the exact sequence $$0\to\h_\xi^d(\m R(\M_J))\to\h_\xi^d(R(\M_J))\to\h_\xi^d(\F_{\m}(\M_J))$$
		of graded $R(J)$-modules.  Since $[\h_\xi^d(R(\M_J))]_n=0$ for all $n\leq 1-d$ by \cite[Theorem 6.2]{trung-gcm} and $[\h_\xi^d(\F_{\m}(\M_J))]_n=0$ for all $n\geq 1-d$, we have 
		$\h_\xi^d(\m R(\M_J))\cong \h_\xi^d(R(\M_J))$ as graded $R(J)$-module. Now the conclusion follows by \cite[Theorem 6.2]{trung-gcm}.
	\end{proof}
	
	\begin{proposition}\label{proposition-2.7}
     Let $M$ be a finitely generated faithful flat $A$-module and $\M=\M_L$ or $\M=\M_\I.$  Let $S_J(\M)\neq 0$. Suppose $(C_1)$ and $(C_2)$ conditions are satisfied by $a_1,\ldots,a_d$ for both $A$ and $M$; also both $A$ and $M$ satisfy condition $(C_3).$   Let $l=g\text{-}\depth_\xi S_J(\M).$ Then (i) $g\text{-}\depth_\xi G(\M)=l-1$
		if $l<d$ and (ii) $S_J(\M)$ is generalized Cohen-Macaulay $R(J)$-module if and only if $g\text{-} \depth_\xi G(\M)\geq d-1$. The latter is the case when $l=d$.
	\end{proposition}
	\begin{proof}
		Clearly $l\leq \dim_{R(J)} S_J(\M)=d$. By Lemma \ref{lemma-2.6(a)} and Lemma \ref{lemma-2.6(b)}, $I_1 R(\M_J)$ and $\m R(\M_J)$ are $(d+1)$ dimensional generalized Cohen-Macaulay $R(J)$-modules respectively.
	 Consider the exact sequences corresponding to the  filtration $\M=\M_\I$ and $\M=\M_L$, respectively
     \begin{align*}
         &0\to I_1 R(\M_J)\to R(\M)_+(1)\to S_J(\M)\to 0,\\
         &0\to \m R(\M_J)\to R(\M)_+(1)\to S_J(\M)\to 0.
     \end{align*}
     Now for any of the filtration $\M=\M_\I$ or $\M=\M_L$, if $l=d$ then it follows that $g\text{-} \depth_\xi R(\M)_+(1)\geq d$ and if $l<d$ then   $g\text{-} \depth_\xi R(\M)_+(1)=l$. 
		Next, we consider the exact sequence 
		$$0\to R(\M)_+\to R(\M)\to M\to 0.$$ Since $M$ is generalized Cohen-Macaulay $A$-module of dimension $d$, we get that 
		$g\text{-} \depth_\xi R(\M)\geq d$ if $l=d$ and $g\text{-} \depth_\xi R(\M)=g\text{-} \depth_\xi R(\M)_+=l$ if $l<d$. By Proposition \ref{Appendix-proposition-main-1},  $g\text{-}\depth_\xi G(\M)=l-1$ if $l<d$.
		
		Finally, we see that $l=d\Longleftrightarrow g \text{-}\depth_\xi R(\M)\geq d \Longleftrightarrow g\text{-}\depth_\xi G(\M)\geq d-1.$ Note that  
		$S_J(\M)$ is generalized Cohen-Macaulay $R(J)$-module if and only if $l=d$. 		
	\end{proof}
    \section{\texorpdfstring
{Buchsbaumness of $G(\I)$}
{Buchsbaumness of G(I)}
}\label{Section-main-1}
 In this section, we develop the proofs of Theorem \ref{t1} and Theorem \ref{t2}. We begin with the following note. 
 
 Consider the canonical homomorphism of 
 graded modules $G(\M) \xrightarrow{\phi} G(\overline{\M})\to 0$.  It is easy to see that $[\ker \phi]_n\cong (M_n\cap W)/(M_{n+1}\cap W)$ which is $(0)$ for $n\gg 0$. 
  So,  $\h^0_{\xi}(\ker\phi)=\ker\phi$ and
 $\h^i_{\xi}(\ker\phi)=0$ for $i\geq 1$.  We get an exact sequence
 \begin{eqnarray}\label{new-eq-1}
0\to\h^0_{\xi}(\ker\phi)\to\h^0_{\xi}(G(\M))\to\h^0_{\xi}(G(\overline{\M})) \mbox{ and }\nonumber \\
\h^i_{\xi}(G(\M))\cong\h^i_{\xi}(G(\overline{\M} )) \mbox{ for } i\geq 1. 
 \end{eqnarray} 
 Suppose $e_1(\M )-e_1^J(M)=2e_0(\M)-2\ell(M/M_1)-\ell(M_1/(M_2+JM))$. Then passing to $\overline{M}$, we may assume that $(C_1)$, $(C_2)$ and $(C_3)$ are satisfied. Then by Corollary \ref{corollary-3.6}, $\depth G(\overline{\M})>0$ which gives 
 \begin{equation}\label{new-eq-2}
\h^0_{\xi}(\ker\phi)=\h^0_{\xi}(G(\M)).
 \end{equation}
 
 \begin{proposition}\label{theorem-5.2}
  Let $M$ be a finitely generated faithful flat $A$-module and $\M=\M_L$ or $\M=\M_\I.$ Suppose $(C_1)$ and $(C_2)$ conditions are satisfied by $a_1,\ldots,a_d$ for both $A$ and $M.$ Let  $e_1(\M)-e_1^J(M)=2e_0(\M)-2\ell_A(M/M_1)-\ell_A(M_1/(M_2+JM))$ holds. Then $G(\M)$ and  $S_J(\M)$ are a generalized Cohen-Macaulay $R(J)$-modules.
 \end{proposition}
 \begin{proof}
 	By Remark \ref{remark-5.1}, we may assume that $M_2\nsubseteq JM$ and $S_J(\M)\neq 0$. We have $\h_{\xi}^i(G(\M))\cong \h_{\xi}^i(G(\overline{\M}))$ for $i\geq 1$ as in \eqref{new-eq-1}.  Now by passing to the ring $\overline{A}$ and module $\overline{M}$, we may assume that the conditions $(C_1)$ and $(C_2)$ are satisfied by the images of $a_1,\ldots,a_d$ in $\overline{A}$ for both $\overline{A}$ and $\overline{M}$ see Lemma \ref{lemma-reduction-steps}. Moreover, $\overline{A}$ and $\overline{M}$ both satisfy $(C_3)$. By Proposition \ref{proposition-3.5}, we have that
 	$$\h_\xi^i(G(\M))\hookrightarrow \h_{\xi}^i(R(\M)/(R(M)_+(1)+R(\M_J)))\hookrightarrow \h_{\xi}^i(S_J(\M))(-1)$$
 	for $0\leq i\leq d-1$. Then $g\text{-}\depth_\xi S_J(\M)\leq g\text{-}\depth_\xi G(\M)$. Therefore, by Proposition 
 	\ref{proposition-2.7}, $g\text{-} \depth_\xi S_J(\M)=g\text{-} \depth_\xi G(\M)=d$ which gives that $G(\M)$ and $S_J(M)$ are generalized Cohen-Macaulay $R(J)$-modules.
 \end{proof}
We now describe the local cohomology modules of $G(\M)$.
\begin{theorem}\label{proposition-5.3}
  Let $M$ be a finitely generated faithful flat $A$-module and $\M=\M_L$ or $\M=\M_\I.$ Suppose $(C_1)$ and $(C_2)$ conditions are satisfied by $a_1,\ldots,a_d$ for both $A$ and $M.$
 Let $e_1(\M)-e_1^J(M)=2e_0(\M)-2\ell(M/M_1)-\ell(M_1/(M_2+JM))$ holds. Then the following statements hold.
	\begin{enumerate}[label=(\roman*)]
		\item\label{Proposition-5.3-item1} For all $n\in\mathbb{Z},$
		$$[\h^0_\xi(G(\M))]_n\cong \begin{cases}
			W/(M_3\cap W) &\mbox{ if } n=2\\
			(M_n\cap W)/(M_{n+1}\cap W) &\mbox{ if } n\geq 3\\
			(0) & \mbox{ otherwise. } 
		\end{cases}    
		$$
		\item\label{Proposition-5.3-item2} We have $\h^i_\xi(G(\M))=[\h^i_\xi(G(\M))]_{2-i}\cong \h_\m^i(M)$ for $1\leq i\leq d-1.$
		\item\label{Proposition-5.3-item3} We have $a(G(\M))\leq 2-d.$
	\end{enumerate}
\end{theorem}
\begin{proof}
	In view of Remark \ref{remark-5.1}, we may assume that $M_2\nsubseteq JM$,  $S_J(\M)\neq (0)$.  By  \eqref{new-eq-2}, 
	$[\h^0_\xi(G(\M))]_n=(M_n\cap W)/(M_{n+1}\cap W)$ which vanishes for $n=0,1$  as  $W \subseteq M_2$ from Proposition \ref{theorem-4.3}. Further, $[\h^0_\xi(G(\M))]_2=W/(M_3\cap W)$. This proves part \ref{Proposition-5.3-item1}. 
    
	 Now by passing to the ring $\overline{A}$ and module $\overline{M}$, we may assume that the conditions $(C_1)$ and $(C_2)$ are satisfied by the images of $a_1,\ldots,a_d$ in $\overline{A}$ for both $\overline{A}$ and $\overline{M}$ see Lemma \ref{lemma-reduction-steps}. Moreover, $\overline{A}$ and $\overline{M}$ both satisfy $(C_3)$, $\depth G(\M)>0$  by Corollary \ref{corollary-3.6} and $f_1=a_1t$ is a non-zero-divisor on $G(\M)$.  
	The exact sequence $0\to G(\M)(-1)\xrightarrow{f_1}G(\M)\to G(\M)/f_1G(\M)\to 0$
	gives the following exact sequence of local cohomology modules
	\begin{align}
		&0\to \h^0_{\xi}(G(\M)/f_1G(\M))\to \h^1_{\xi}(G(\M))(-1)\xrightarrow{f_1}\h^1_{\xi}(G(\M))\to\h^1_{\xi}(G(\M)/f_1G(\M))\to\ldots \nonumber \\
		&\to\h^{i-1}_{\xi}(G(\M)/f_1G(\M))\to \h^{i}_{\xi}(G(\M))(-1)\xrightarrow{f_1} \h^{i}_{\xi}(G(\M)) \to \h^{i}_{\xi}(G(\M)/f_1G(\M))\to\ldots\nonumber\\
		&\to\h^{d-1}_{\xi}(G(\M)/f_1G(\M))\to \h^d_{\xi}(G(\M))(-1)\xrightarrow{f_1}  \h^d_{\xi}(G(\M))\to 0. \label{exact-seq-LC} 
	\end{align}
	We apply induction on $d$  to prove the assertions in  \ref{Proposition-5.3-item2} and \ref{Proposition-5.3-item3}. 
	Suppose $d=1$. Then $M_{n+2}=J^nM_2$ for $n\geq 1$ by Proposition \ref{theorem-4.3}. This gives $G(\M)/f_1 G(\M)=M/M_1\oplus M_1/(M_2+JM)\oplus M_2/JM_1$ and $[\h^1_{\xi}(G(\M))]_n=0$ for all $n\gg 0$. This gives that $a_1(G(\M))\leq 1$.
	
	Now let $d\geq 2$ and the assertions in \ref{Proposition-5.3-item2} and \ref{Proposition-5.3-item3} hold for $d-1$. Let $M^\prime=M/a_1M,$ $\M^\prime=\{(M_n+a_1M)/a_1M\}_{n\geq 0}$ and $\M_J^\prime=\{(J^nM+a_1M)/a_1M\}_{n\geq 0}$. Note that $G(\M^\prime)\cong G(\M)/f_1G(\M).$ Then by Lemma \ref{lemma-reduction-steps}, 
	$$e_1(\M^\prime)-e_1^J(M^\prime)=2e_0(\M^\prime)-2\ell_A(M^\prime/{M}^\prime_1)-
	\ell_A({M}^\prime_1/({M}^\prime_2+JM^\prime))$$ and 
	the conditions $(C_1)$ and $(C_2)$ are satisfied by the images of $a_1,\ldots,a_d$ in $A^\prime$ for both $A^\prime$ and $M^\prime.$ By induction hypothesis, we have that
	$$a_{d-1}(G(\M^\prime))\leq 2-(d-1)=3-d$$
	and  for all $1\leq i\leq d-2,$
	\begin{equation}\label{induction-LC}\h^i_{\xi}(G(\M^\prime))=[\h^i_{\xi}(G(\M^\prime))]_{2-i}\cong\h^i_\m(M^\prime).\end{equation}
	
	For $i=d-1$, we consider the exact sequence, obtained from  \eqref{exact-seq-LC},  $$[\h^{d-1}_{\xi}(G(\M^\prime))]_{a(G(\M))+1}\to  [\h^d_{\xi}(G(\M)(-1))]_{a(G(\M))+1}\to 0$$
	which implies that $[\h^{d-1}_{\xi}(G(\M^\prime))]_{a(G(\M))+1}\neq 0$. Therefore $a(G(\M))+1\leq a(G(\M^\prime))\leq 3-d$ i.e. $a(G(\M))\leq 2-d.$
	
	Now the condition $(C_1)$ implies that $J=(a_1,\ldots,a_d)$ is a standard system of parameters of $M$. Then $(a_2,\ldots,a_d)$ is a standard
	system of parameters of $M^\prime.$ Since  $M_{n+2}=J^n M_2$ for all $n\geq 1$ by Proposition \ref{theorem-4.3}, we have 
	$M^\prime_n\cap \h_{\m}^0(M^\prime)\subseteq JM^\prime\cap \h^0_{\m}(M^\prime)=0$ for all $n\geq 3$ by \cite[Corollary 2.3]{trung-gcm}. Therefore, 
	by part \ref{Proposition-5.3-item1}, 
	\begin{equation}\label{equation:int2}\h^0_{\xi}(G(\M^\prime))=[\h^0_{\xi}(G(\M^\prime))]_2\cong\h^0_\m(M^\prime).\end{equation}
	
	Since
	$\h^{i-1}_{\xi}(G(\M^\prime))=[\h^{i-1}_{\xi}(G(\M^\prime))]_{3-i}$ for all $1\leq i\leq d-1$ by  \eqref{induction-LC} and \eqref{equation:int2}, we get the exact sequences 
	$$0\to [\h^i_{\xi}(G(\M))]_{n-1}\to[\h^i_{\xi}(G(\M))]_n $$
	for all $n\geq 4-i$ and $1\leq i\leq d-1$.
	This gives  $[\h^i_{\xi}(G(\M))]_n=0$ for all $n\geq 3-i$ and $1\leq i\leq d-1.$
	
	From \eqref{exact-seq-LC} and \eqref{induction-LC}, we also get 
	$$[\h^i_{\xi}(G(\M))]_{n-1}\to[\h^i_{\xi}(G(\M))]_n\to 0$$
	for all $n\leq 1-i$ and $1\leq i\leq d-2$. 
	By Proposition \ref{theorem-5.2}, $G(\M)$ is generalized Cohen-Macaulay $R(J)$-module which implies that $[\h^i_{\xi}(G(\M))]_n=0$ for $n\ll 0 $ for $1\leq i\leq d-2$. Therefore, we have 
	$[\h^i_{\xi}(G(\M))]_n=(0)$ for all $1\leq i\leq d-2$ and $n\leq 1-i$. Thus, 	for $1\leq i\leq d-2$,
	$$\h^i_{\xi}(G(\M))=[\h^i_{\xi}(G(\M))]_{2-i}.$$
	Next, we show that $[\h^{d-1}_{\xi}(G(\M))]_n=(0)$ for all $n\leq 2-d$. Considering the monomorphisms
	$$0\to [\h^{d-1}_{\xi}(G(\M))]_{n-1}\to [\h^{d-1}_{\xi}(G(\M))]_n$$
	for all $n\leq 2-d$, it is enough to show that  $[\h^{d-1}_{\xi}(G(\M))]_{2-d}=(0).$ 
	By Proposition \ref{proposition-3.5}, we have the injections
	\begin{equation}\label{equation-int1}
		[\h^{d-1}_\xi(G(\M))]_{2-d}\hookrightarrow [\h^{d-1}_{\xi}(R(\M)/(R(\M)_+(1)+R(\M_J)))]_{2-d}\hookrightarrow [\h^{d-1}_\xi(S_J(\M))]_{1-d}.\end{equation}       
	We show that $[\h^{d-1}_{\xi}(S_J(\M))]_{1-d}=(0)$. 

{\bf{Claim:}} For $1\leq i\leq d-1,$ we have $[\h^i_{\xi}(R(\M)_+(1))]_n=(0)$ for all $n\leq 2-i$ and $n<0$.
\begin{proof}[Proof of Claim] The exact sequence $0\to R(\M)_+(1)\to R(\M)\to G(\M)\to 0$ of graded $R(J)$-modules induces the exact
	sequence $$ \h^{i-1}_{\xi}(G(\M))\to \h^i_{\xi}(R(\M)_+(1))\to \h^i_{\xi}(R(\M))$$
	of local cohomology modules. Since  $\h^{i-1}_{\xi}(G(\M))=[\h^{i-1}_{\xi}(G(\M))]_{3-i}$, we have the injective maps 
	$$[\h^{i}_{\xi}(R(\M)_+(1))]_n\hookrightarrow [\h^{i}_{\xi}(R(\M))]_n$$  for $n\neq 3-i.$ 
	On the other hand, the exact sequence $0\to R(\M)_+\to R(\M)\to M\to 0$ induces isomorphism $[\h^i_{\xi}(R(\M))]_n\cong [\h^i_{\xi}(R(\M)_+)]_{n}$ for all $n<0$.
	Thus,  \begin{equation}\label{equation-int}
		[\h^{i}_{\xi}(R(\M)_+(1))]_n\hookrightarrow [\h^{i}_{\xi}(R(\M))]_n \cong [\h^i_{\xi}(R(\M)_+(1))]_{n-1} 
	\end{equation}
	for all $n\leq 2-i$ and $n<0.$ Since $G(\M)$ is generalized Cohen-Macaulay $R(J)$-module by Proposition
	\ref{theorem-5.2}, $R(\M)$ is generalized Cohen-Macaulay $R(J)$-module by Theorem \ref{Appendix-proposition-main-1} 
which implies that $\h^i_{\xi}(R(\M))$ is finitely graded $R(J)$-module for $1\leq i\leq d-1$. Using \eqref{equation-int}, 
	we get that $[\h^i_{\xi}(R(\M)_+(1))]_n=0$ for all $n\leq 2-i$ and $n<0$.
\end{proof}
Now consider the exact sequences 
\begin{align}\label{equation:int5(a)}
0\to I_1 R(\M_J)\to R(\M)_+(1)\to S_J(\M)\to 0\\
0\to \m R(\M_J)\to R(\M)_+(1)\to S_J(\M)\to 0 \label{equation:int5(b)}
\end{align} corresponding to the filtrations $\M=\M_\I$ and $\M=\M_L$ respectively and the induced exact sequences of local cohomology modules 
$$[\h^{d-1}_{\xi}(R(\M)_+(1))]_{1-d}\to [\h^{d-1}_{\xi}(S_J(\M))]_{1-d}\to [\h^d_{\xi}(I_1 R(\M_J))]_{1-d} \text{ and }$$
$$
[\h^{d-1}_{\xi}(R(\M)_+(1))]_{1-d}\to [\h^{d-1}_{\xi}(S_J(\M))]_{1-d}\to [\h^d_{\xi}(\m R(\M_J))]_{1-d}
$$
respectively.
By Lemma \ref{lemma-2.6(a)}, $[\h^d_{\xi}(I_1 R(\M_J))]_{1-d}=(0)$ and by Lemma \ref{lemma-2.6(b)} $[\h^d_{\xi}(\m R(\M_J))]_{1-d}=(0)$. Since by the above claim $[\h^{d-1}_{\xi}(R(\M)_+(1))]_{1-d}=(0)$, we get 
$[\h^{d-1}_{\xi}(S_J(\M))]_{1-d}=(0)$ and so $[\h^{d-1}_{\xi}(G(\M))]_{2-d}=0$ by \eqref{equation-int1}.

It remains to show that $[\h^i_{\xi}(G(\M))]_{2-i}\cong \h^i_\m(M)$ for $1\leq i\leq d-1.$ 
The exact sequence $0\to M\xrightarrow{a_1} M\to M^\prime\to 0$ induces the isomorphism $\h^0_\m(M^\prime)\cong \h^1_\m(M)$ since 
$a_1\h^1_\m(M)=(0).$ Also, from \eqref{exact-seq-LC}, we have the exact sequence
$$0\to [\h ^0_{\xi}(G(\M^\prime))]_{2}\to [\h^0_{\xi}(G(\M))]_{1} \to [\h^0_{\xi}(G(\M))]_{2}=0 $$
Using \eqref{equation:int2}, we get that 
$$\h^1_{\xi}(G(\M))=[\h^1_{\xi}(G(\M))]_{1}\cong [\h^0_{\xi}(G(\M^\prime))]_{2}\cong \h^0_\m(M^\prime)\cong \h^1_\m(M).$$ 
Suppose $2\leq i\leq d-1.$  From Proposition \ref{proposition-3.5},
we have
$$[\h^i_{\xi}(G(\M))]_{2-i}\hookrightarrow [\h^i_{\xi}(R(\M)/(R(\M)_+(1)+R(\M_J)))]_{2-i}\hookrightarrow [\h^i_{\xi}(S_J(\M))]_{1-i}.$$ From \eqref{equation:int5(a)}, we have
$$[\h^{i}_{\xi}(R(\M)_+(1))]_{1-i}\to [\h^{i}_{\xi}(S_J(\M))]_{1-i}\to [\h^{i+1}_{\xi}(I_1 R(\M_J))]_{1-i}$$ and from \eqref{equation:int5(b)}, we have
$$[\h^{i}_{\xi}(R(\M)_+(1))]_{1-i}\to [\h^{i}_{\xi}(S_J(\M))]_{1-i}\to [\h^{i+1}_{\xi}(\m R(\M_J))]_{1-i}.$$
Since $[\h^{i}_{\xi}(R(\M)_+(1))]_{1-i}=0$ by the above claim , $[\h^{i+1}_{\xi}(I_1 R(\M_J))]_{1-i}\cong \h^i_\m(M)$ by Lemma \ref{lemma-2.6(a)} and $[\h^{i+1}_{\xi}(\m R(\M_J))]_{1-i}\cong \h^i_\m(M)$ by Lemma \ref{lemma-2.6(b)}, we get that
$$\h^i_{\xi}(G(\M))=[\h^i_{\xi}(G(\M))]_{2-i}\hookrightarrow \h^i_\m(M)$$
for $2\leq i\leq d-1.$ Then 
\begin{equation*}
	\mathbb{I}(M)\leq \mathbb{I}(G(\M))=\mathop\sum\limits^{d-1}_{i=1}\binom{d-1}{i}\ell(\h^i_{\xi}(G(\M)))\leq \mathop\sum\limits^{d-1}_{i=1}\binom{d-1}{i}\ell_A(\h^i_\m(M))=\mathbb{I}(M)
\end{equation*}
see \cite[Corollary 5.2]{trung-gcm} for the first inequality. Then  
$\h^i_{\xi}(G(\M))=[\h^i_{\xi}(G(\M))]_{2-i}\cong\h^i_\m(M)$ for $1\leq i\leq d-1.$ This completes the proof. 
\end{proof}
Now we state and prove our main results. 
\begin{theorem}\label{m1}

	 Let $M$ be a finitely generated faithful flat $A$-module and $\M=\M_L$ or $\M=\M_\I.$ Suppose $(C_1)$ and $(C_2)$ conditions are satisfied by $a_1,\ldots,a_d$ for both $A$ and $M.$ Suppose the equality \begin{equation*}e_1(\M)-e_1^J(M)=2e_0(\M)-2\ell(M/M_1)-\ell(M_1/(M_2+JM))
	\end{equation*}  holds. Then $G(\M)$ is generalized Cohen-Macaulay $R(J)$-module with $$\ell(\h^0_{\xi}(G(\M)))=\ell(\h^0_{\m}(M)),~ \h^i_{\xi}(G(\M))=[\h^i_{\xi}(G(\M))]_{2-i}\cong \h^i_{\m}(M)$$ for $1\leq i\leq d-1$ and $a(G(\M))\leq 2-d.$ Furthermore, 
	\begin{enumerate}
	 \item\label{the-theorem-implicaitons-cond-4} $e_2(\M)=e_1^J(M)+e_2^J(M)+e_1(\M)-e_0(\M)+\ell_A(M/M_1)$ \mbox{if } $d\geq 2$;
	\item\label{the-theorem-implicaitons-cond-5} $e_i(\M)=e_{i-2}^J(M)+2e_{i-1}^J(M)+e_i^J(M)$ \mbox{for } $3\leq i\leq d$. 
		\end{enumerate}
\end{theorem}
When this is the case, clearly $\mathbb{I}(G(\M))=\mathbb{I}(M)$ and $\depth G(\M)=\depth M$. 
\begin{proof}
	By Proposition \ref{theorem-5.2}, $G(\M)$ is generalized Cohen-Macaulay $R(J)$-module and the description of the local cohomology modules  $\h^i_{\xi}(G(\M))$ follows from Theorem \ref{proposition-5.3}. It remains to find the Hilbert coefficients. 	If $M_2\subseteq JM$, then $M$ is Cohen-Macaulay $A$-module and $M_{n+1}=J^nM_1$ for $n\geq 1$ by Remark \ref{remark-5.1}. Then we have 
	$e_1^J(M)=0$ and $e_1(\M)=2e_0(\M)-2\ell_A(M/M_1)-\ell_A(M_1/JM)=e_0(\M)-\ell_A(M/M_1)$ which implies $e_i(\M)=0$ for $2\leq i\leq d$, see \cite[Theorem 2.9]{rossi-valla}. 
	Since $e_i^J(M)=0$ for $1\leq i\leq d$, the  assertions (\ref{the-theorem-implicaitons-cond-4}) and (\ref{the-theorem-implicaitons-cond-5}) hold. 
	
	Suppose $M_2\nsubseteq JM.$ We have $e_i(\overline{\M})=e_i(\M)$, $e_i^{J\overline{A}}(\overline{M})=e_i^J(M)$ for $1\leq i\leq d-1$, $e_d(\M)=e_d(\overline{\M})+(-1)^d\ell_A(W)$, $e_d^J(M)=e_d^{J\overline{A}}(\overline{M})+(-1)^d\ell_A(W)$ 
	and $\ell_A(M/M_1)=\ell_A(\overline{M}/\overline{M_1})$ as $W\subseteq M_2$ by Proposition \ref{theorem-4.3}. Then  $e_1(\overline{\M})-e_1^{J\overline{A}}(\overline{M})=2e_0(\overline{\M})-2\ell_A(\overline{M}/\overline{M_1} )-\ell_A(\overline{M_1} /(\overline{M_2}+J\overline{M}))$, where $\overline{\M}$ denote the $I\overline{A}$-good filtration $\{ \overline{M_n}=(M_n+W)/W\}_{n\geq 0}$ of $\overline{M}$.
	 Therefore, by passing to the ring $\overline{A}$ and module $\overline{M}$, we may assume that the conditions $(C_1)$ and $(C_2)$ are satisfied by the images of $a_1,\ldots,a_d$ in $\overline{A}$ for both $\overline{A}$ and $\overline{M}$ see Lemma \ref{lemma-reduction-steps}. Moreover, $\overline{A}$ and $\overline{M}$ both satisfy $(C_3)$. By Corollary \ref{corr-2.5}, 
\begin{equation}\label{new-eqn-3}
	e_i(\M)=e_{i-1}^J(M)+e_i^J(M)+e_{i-1}(S_J(\M))
\end{equation}
	for $2\leq i\leq d.$  For $n\gg 0$, we write 
		\begin{align*}
		\ell_A(S_J(\M)_{n-1})=&\mathop\sum\limits_{i=0}^{d-1}(-1)^i e_i(S_J(\M))\binom{n+d-2-i}{d-1-i}\\
		=&e_0(S_J(\M))\binom{n+d-1}{d-1}+\mathop\sum\limits_{i=1}^{d-1}(-1)^i(e_i(S_J(\M))+\\&\qquad e_{i-1}(S_J(\M)))\binom{n+d-1-i}{d-1-i}.
	\end{align*}
	By Theorem \ref{lemma-3.3} and Proposition \ref{proposition-3.5}, we have 
	\begin{align*}
		\ell_A(S_J(\M)_{n-1})&= \ell_A\Big(\frac{M_n}{M_{n+1}}\Big)-\Big\{\ell_A\Big(\frac{M}{M_1}\Big)+\ell_A\Big(\frac{M_1}{M_2+JM}\Big)\Big\}\binom{n+d-1}{d-1}
		 + \\
         &\qquad \ell_A\Big(\frac{M_1}{M_2+JM}\Big)\binom{n+d-2}{d-2}
	\end{align*}
	for all $n\geq 0.$ On comparing the coefficients, we get 
	\begin{align*}
		e_0(S_J(\M))=e_0(\M)-\ell_A(M/M_1)-\ell_A(M_1/(M_2+JM)),\\
		e_1(S_J(\M))+e_0(S_J(\M))=e_1(\M)-\ell_A(M_1/(M_2+JM)) \text{ and }\\
		e_i(S_J(\M))+e_{i-1}(S_J(\M))=e_i(\M) \mbox{ for } 2\leq i\leq d-1.
	\end{align*}
Hence $e_1(S_J(\M))=e_1(\M)-e_0(S_J(\M))-\ell_A(M_1/(M_2+JM))=e_1(\M)-e_0(\M)+\ell_A(M/M_1)$ and  $e_i(S_J(\M))=e_i(\M)-e_{i-1}(S_J(\M))=e_{i-1}^J(M)+e_i^J(M)$,	for $2\leq i\leq d-1$. The last part follows from  Corollary \ref{corr-2.5}. We put these values in \eqref{new-eqn-3} to complete the proof.
\end{proof}
On restricting to the case of ideal filtration, we immediately obtain Theorem \ref{Corollary 7.4} and Theorem \ref{m2}.
\begin{corollary}\label{Corollary 7.4}
    Let $(A,\m)$ be a Noetherian local ring, $\I=\{I_n\}_{n\geq 0}$ be a multiplicative $I$-good filtration of $A$ and $J=(a_1,\ldots,a_d)\subseteq I$  a reduction of $\I$  such that the conditions $(C_1)$ and $(C_2)$ are satisfied by $a_1,\ldots,a_d$ for $A$. Suppose the equality \begin{equation*}e_1(\I)-e_1(J)=2e_0(\I)-2\ell_A(A/I_1)-\ell_A(I_1/(I_2+J))
	\end{equation*}  holds. Then $G(\I)$ is generalized Cohen-Macaulay with $$\ell_A(\h^0_{\xi}(G(\I)))=\ell_A(\h^0_{\m}(A)),~ \h^i_{\xi}(G(\I))=[\h^i_{\xi}(G(\I))]_{2-i}\cong \h^i_{\m}(A)$$ for $1\leq i\leq d-1$ and $a(G(\I))\leq 2-d.$ Furthermore, 
	\begin{enumerate}
	 \item
     $e_2(\I)=e_1(J)+e_2(J)+e_1(\I)-e_0(\I)+\ell_A(A/I_1)$ \mbox{if } $d\geq 2$;
	\item
    $e_i(\I)=e_{i-2}(J)+2e_{i-1}(J)+e_i(J)$ \mbox{for } $3\leq i\leq d$. 
		\end{enumerate}
        When this is the case, clearly $\mathbb{I}(G(\I))=\mathbb{I}(A)$ and $\depth G(\I)=\depth A$.
\end{corollary}

In a Buchsbaum local ring, having $\mathbb{I}(G(\I))=\mathbb{I}(A)$ is a sufficient condition to conclude that $G(\I)$ is Buchsbaum. We recall the following theorem  of \cite{sal20} for our use, also see \cite{ozeki13} for $I$-adic case. 
\begin{theorem}\cite[Theorem 1.2]{sal20}\label{theorem-of-sal20}
Let $(A,\m)$ be a Buchsbaum local ring and $J=(a_1,\ldots,a_d)\subseteq I$ a reduction of $\I$. Then the following statements are equivalent:
\begin{enumerate}
	\item $G(\I)$ is a Buchsbaum $R(\I)$-module and $\ell_A(\h^i_{\xi}(G(\I)))=\ell_A(\h^i_\m(A))$ for $0\leq i< d.$
	\item $\mathbb{I}(G(\I))=\mathbb{I}(A)$.
\end{enumerate}
\end{theorem}
Now we state our result on Buchsbaumness of $G(\I)$.
\begin{theorem}\label{m2}
	Let $(A,\m)$ be a Buchsbaum local ring and   $J=(a_1,\ldots,a_d)\subseteq I$ a reduction of $\I$  such that the condition $(C_2)$ is satisfied by $a_1,\ldots,a_d$ for $A$. Suppose \begin{equation*}e_1(\I)-e_1(J)=2e_0(\I)-2\ell_A(A/I_1)-\ell_A(I_1/(I_2+J)).
	\end{equation*} Then $G(\I)$ is Buchsbaum with $\mathbb{I}(G(\I))=\mathbb{I}(A)$.
\end{theorem}
\begin{proof}
	Suppose $A$ is Buchsbaum. Then the condition $(C_1)$ is satisfied and  $\mathbb{I}(G(\I))=\mathbb{I}(A)$ by Corollary \ref{Corollary 7.4}. Therefore,  by Theorem \ref{theorem-of-sal20}, $G(\I)$ is Buchsbaum.
\end{proof}
\section{\texorpdfstring
{Buchsbaumness of $F_{\m}(I)$}
{Buchsbaumness of Fm(I)}
}\label{Section-fiber-cone}

Now we discuss the passage of Buchsbaumness from the local ring to the special fiber ring  $F_\m(I)$ of $I$.  We prove Theorems \ref{theorem-fiber-1} and \ref{theorem-fiber-2} in this section which are the module versions of Theorem 
\ref{theorem-fiber-m1} and Theorem \ref{theorem-fiber-m2}.  Recall that $\mathcal{U}:=\m R(I)+R(I)_+$.  For the rest of the paper,  we restrict ourselves to $\I=\{I^n\}_{n\geq0}$ and $\M=\{I^nM\}_{n\geq 0}$. We write $f_0^I(M)$ for the fiber multiplicty of the module $F_\m(\M)=\mathop\oplus\limits_{n\geq 0} I^nM/\m I^nM.$ Let
 $J=(a_1,\ldots,a_d)\subseteq I$ be a minimal reduction of $\I$. Then, for $n\geq 0$, $$\ell_A(I^nM/\m I^nM)=\ell_A(M/\m I^nM)-\ell_A(M/I^nM)$$  which gives the relation $f_0^I(M)=e_1^I(M)-e_1(\M_L)+e_0^I(M)$. Therefore, 
\begin{equation*}
	\begin{split}
		e_1(\M_L)-e_1^J(M)=2e_0^I(M)-2\ell_A(M/\m M)-\ell_A(\m M/(I\m M+JM))\mbox{ if and only if} \\
		f_0^I(M)= e_1^I(M)-e_0^I(M)-e_1^J(M)+\ell_A(M/\m M)+\ell_A(M/(I\m M+JM)). 
	\end{split}
\end{equation*}
Further, we consider the following exact sequences of $R(I)$-modules. 
\begin{eqnarray}
	0\to N\to G(\M)\to F_\m(\M)\to 0 \label{1st-exact-eqn}\\
	0\to F_\m(\M)\to G(\M_L)\to N(-1)\to 0\label{2nd-exact-eqn}
\end{eqnarray}
where $N=\mathop\oplus\limits_{n\geq 0}\m I^nM/I^{n+1}M$. This induces the exact sequences of local cohomology modules
\begin{eqnarray}
	0\to [\h^0_{\mathcal{U}}(N)]_n\to\ldots\to [\h^i_{\mathcal{U}}(G(\M))]_n\to [\h^i_{\U}(F_\m(\M))]_n\to [\h^{i+1}_{\U}(N)]_n\to\ldots \text{ and }\label{long-exact-seq-1}\\
	0\to[\h^0_{\U}(F_\m(\M))]_n\to\ldots
	\to [\h^i_{\mathcal{U}}(G(\M_L))]_n\to [\h^i_{\mathcal{U}}(N)]_{n-1}\to[\h^{i+1}_{\U}(F_\m(\M))]_n\to\ldots\label{long-exact-seq-2}
\end{eqnarray}

\begin{lemma}\label{fiber-lemma-1}
 Let $M$ be a finitely generated faithful flat $A$-module. Suppose $(C_1)$ and $(C_2)$ conditions are satisfied by $a_1,\ldots,a_d$ for both $A$ and $M.$ Suppose
 \begin{equation}\label{the-theorem-2-condition1}
   e_1^I(M)-e_1^J(M)=2e_0^I(M)-2\ell_A(M/IM)-\ell_A(IM/(I^2M+JM)) \text{and}  
 \end{equation}
 \begin{equation}\label{the-theorem-2-condition2}
     f_0^I(M)= e_1^I(M)-e_0^I(M)-e_1^J(M)+\ell_A(M/\m M)+\ell_A(M/(I\m M+JM)).
 \end{equation}
  Then  
 \begin{enumerate}
 	\item\label{1} $$[\h^0_\mathcal{U}(N)]_n\cong \begin{cases}
 		0 &\mbox{ if } n= 0,1\\
 		(\m I^nM\cap W)/(I^{n+1}M\cap W)
 		&\mbox{ if } n\geq 2.
 	\end{cases} $$
\item \label{2}
 $$[\h^0_\mathcal{U}(F_\m(\M))]_n\cong \begin{cases}
 	0 &\mbox{ if } n= 0,1\\
 	W/(\m I^2M\cap W) &\mbox{ if } n=2\\
 	(I^nM\cap W)/(\m I^{n}M\cap W) &\mbox{ if } n\geq 3.\\
\end{cases}$$
\end{enumerate}
Furthermore, $\h^1_\mathcal{U}(F_{\m}(\M))=[\h^1_\mathcal{U}(F_\m(\M))]_{1}\cong \h_\m^1(M)$ and $\h^1_\U(N)=[\h^1_\U(N)]_1\hookrightarrow \h^1_\m(M).$ 
\end{lemma}
\begin{proof}
		From the exact sequences \eqref{long-exact-seq-1} and \eqref{long-exact-seq-2}, we have 
		$$[\h^0_\mathcal{U}(N)]_{n}=\ker \Big( [\h^0_\mathcal{U}(G(\M))]_{n}\to [\h^0_\mathcal{U}(F_\m (\M))]_{n}\Big )\text{ and }$$  $$[\h^0_\mathcal{U}(F_\m (\M))]_{n}= \ker\Big( [\h^0_\mathcal{U}(G (\M_L))]_{n}\to \h^0_\mathcal{U}(N(-1))]_{n}\Big).$$
		
		By Theorem \ref{proposition-5.3}, $[\h^0_\mathcal{U}(F_\m(\M))]_n=0=[\h^0_\mathcal{U}(N)]_n$ for $n=0,1$ and $[\h^0_\mathcal{U}(F_\m(\M))]_2\cong [\h^0_\mathcal{U}(G(\M_L))]_2\cong W/(\m I^2M\cap W)$. Further, $[\h^0_\mathcal{U}(N)]_2\cong \ker \Big( W/(I^3M\cap W)\to W/(\m I^2M\cap W)\Big )=(\m I^2M\cap W)/(I^3M\cap W)$ which implies $[\h^0_\mathcal{U}(F_\m(\M))]_3\cong \ker \Big((\m I^2M\cap W)/(\m I^3M\cap W)\to (\m I^2M\cap W)/(I^3M\cap W)\Big)=(I^3M\cap W)/(\m I^3M\cap W)$. We may now proceed by induction on $n$. Suppose $[\h^0_\mathcal{U}(N)]_{n-1}\cong 	(\m I^{n-1}M\cap W)/(I^{n}M\cap W)$ and  $[\h^0_\mathcal{U}(F_\m(\M))]_n\cong (I^nM\cap W)/(\m I^nM\cap W)$ for some $n\geq 3$. Then  $[\h^0_\mathcal{U}(N)]_{n}\cong \ker \Big( (I^nM \cap W)/(I^{n+1}M\cap W)\to (I^nM\cap W)/(\m I^nM\cap W) \Big)=(\m I^nM \cap W)/(I^{n+1}M\cap W)$ which implies that 
		$[\h^0_\mathcal{U}(F_\m(\M))]_{n+1}\cong \ker \Big( (\m I^nM\cap W)/(\m I^{n+1}M\cap W)\to (\m I^nM\cap W)/(I^{n+1}M\cap W)  \Big)=(I^{n+1}M\cap W)/(\m I^{n+1}M\cap W).$
        We now prove that $$\h^1_\mathcal{U}(F_{\m}(\M))=[\h^1_\mathcal{U}(F_\m(\M))]_{1}\cong \h_\m^1(M).$$ By the exact sequence \eqref{long-exact-seq-2}, we have
		$0=[\h^0_\mathcal{U}(N)]_{n-1}\to [\h^1_\mathcal{U}(F_{\m}(\M))]_n\to [\h^1_\mathcal{U}(G(\M_L))]_n\to [\h^1_\mathcal{U}(N(-1))]_{n}$ for $n=0,1,2$. Using Theorem \ref{proposition-5.3}, we get that $[\h^1_\mathcal{U}(F_{\m}(\M))]_n=0$ for $n=0,2$ and $[\h^1_\mathcal{U}(F_{\m}(\M))]_1\cong \ker \Big([\h^1_\mathcal{U}(G(\M_L))]_1\cong \h_\m^1(M)\to [\h^1_\mathcal{U}(N(-1))]_1 \Big)$ where $[\h^1_\mathcal{U}(N(-1))]_1=0$ follows from the exact sequence $0=[\h^0_\mathcal{U}(F_{\m}(\M))]_0\to [\h^1_\mathcal{U}(N)]_0\to [\h^1_\mathcal{U}(G(\M))]_0=0$.
		Therefore, $$[\h^1_\mathcal{U}(F_{\m}(\M))]_1=[\h^1_\mathcal{U}(G(\M_L))]_1\cong\h_\m^1(M).$$ Further by the exact sequence \eqref{long-exact-seq-2} and Theorem \ref{proposition-5.3}, we have exact sequences 
		$$0\to \frac{W\cap I^nM}{W\cap \m I^nM}\to \frac{W\cap \m I^{n-1}M}{W\cap \m I^nM}\to \frac{W\cap \m I^{n-1}M}{W\cap I^nM}\to [\h^1_\mathcal{U}(F_{\m}(\M))]_n\to 0$$
		for $n\geq 3$. Therefore, $[\h^1_\mathcal{U}(F_{\m}(\M))]_n=0$ for $n\geq 3.$
        
        Now we want to prove that $\h^1_\U(N)=[\h^1_\U(N)]_1.$ To see this, consider the exact sequence 
		$0=[\h^0_{\mathcal{U}}(F_\m(\M))]_n\to [\h^1_{\mathcal{U}}(N)]_n\to [\h^1_{\mathcal{U}}(G(\M))]_n=0$ for $n\leq 0$ which gives $[\h^1_{\mathcal{U}}(N)]_n=0$ for $n\leq 0.$ For $n=1$, we have $0\to [\h^1_{\mathcal{U}}(N)]_1\to [\h^1_{\mathcal{U}}(G(\M))]_1\cong \h^1_\m(M)\to [\h^1_{\mathcal{U}}(F_\m(\M))]_1.$ Thus,  $[\h^1_{\mathcal{U}}(N)]_1 \hookrightarrow \h^1_\m(M).$ For $n\geq 2$, we have 
		$$0\to  [\h^0_{\mathcal{U}}(N)]_n\to  [\h^0_{\mathcal{U}}(G(\M))]_n\to  [\h^0_{\mathcal{U}}(F_\m(\M))]_n\to  [\h^1_{\mathcal{U}}(N)]_n\to 0.$$
		which gives $W/(I^3M\cap W)\to W/(\m I^2M\cap W)\to [\h^1_{\mathcal{U}}(N)]_2\to 0$ and $(I^nM\cap W)/(I^{n+1}M\cap W)\to (I^nM\cap W)/(\m I^nM\cap W)\to[\h^1_{\mathcal{U}}(N)]_n\to 0$ for $n\geq 3.$
		This implies $[\h^1_{\mathcal{U}}(N)]_n=0$ for $n\geq 2.$ 
	\end{proof}
\begin{lemma}\label{fiber-lemma-2} With the same hypothesis as in Lemma \ref{fiber-lemma-1}, we have for all $ 1\leq i \leq d-1$,  
    $[\h^i_\U(F_\m(\M))]_n=0=[\h^i_\U(N)]_n$ except for at most $i$-many values of $n$. 
    \end{lemma}
\begin{proof}
   We induct on $ i$. By Lemma \ref{fiber-lemma-1}, $\h^1_\U(F_\m(\M))$ has at most one non-zero graded piece and also $\h^1_\U(N)$ has at most one non-zero graded piece. Now let $1\leq i\leq d-2$ and the assertion be true for $i$. 
 

    Now  for $\h^{i+1}_\mathcal{U}(F_\m(\M))$, consider the following exact sequence obtained from \eqref{long-exact-seq-2}, 
    \begin{equation}\label{8.2-induced}
        [ \h^i_{\mathcal{U}}(N)]_{n-1}\xrightarrow{f}[\h^{i+1}_{\U}(F_\m(\M))]_n\xrightarrow{g} [\h^{i+1}_{\mathcal{U}}(G(\M_L))]_n.
    \end{equation}    
    Since by Theorem \ref{proposition-5.3}, $[\h^{i+1}_{\mathcal{U}}(G(\M_L))]_n=0$ for $n\neq 1-i$ and by induction hypothesis $[\h^i_{\mathcal{U}}(N)]_{n-1}\neq 0$ for at most $i$-many values of $n$. Therefore $[\h^{i+1}_{\U}(F_\m(\M))]_n\neq0$ for at most $i+1$ many values of $n.$
    Now for $\h_\U^{i+1}(N)$, consider the following exact sequence obtained from \eqref{long-exact-seq-1},  
    \begin{equation}\label{8.3-induced}
      [\h^i_{\U}(F_\m(\M))]_n\xrightarrow{p}[\h^{i+1}_{\U}(N)]_n\xrightarrow{q}[\h^{i+1}_{\mathcal{U}}(G(\M))]_n.   
    \end{equation}
    Since $[\h^{i+1}_\U(G(\M))]_n=0$ for $n\neq 1-i$ and $[\h^i_\U(F_\m(\M))]_n\neq 0$ for at most $i$ values of $n.$ Hence $[\h^{i+1}_\U(N)]_n\neq 0$ for at most $i+1$ values of $n.$ 
\end{proof}
\begin{lemma}\label{fiber-lemma-3}  With the same hypothesis as in Lemma \ref{fiber-lemma-1}, we have
$\ell_A\Big([\h^i_\mathcal{U}(F_\m(\M))]_n\Big)< \infty$ for all $n$ and for all $ 0\leq i \leq d-1.$
\end{lemma}
\begin{proof}
 By Lemma \ref{fiber-lemma-1}, $\ell_A\Big([\h^i_\U(F_\m(\M))]_n\Big)<\infty$, $\ell_A\Big([\h^i_\U(N)]_n\Big)<\infty$ for $i=0,1$ and for all $ n.$ 
 Now suppose $1\leq i\leq d-2$ and $\ell_A\Big([\h^i_\U(F_\m(\M))]_n\Big)<\infty$ and $\ell_A\Big([\h^i_\U(N)]_n\Big)<\infty$ for all $n$. Consider the exact sequence 
    $$
    0 \to \frac{[ \h^i_{\mathcal{U}}(N)]_{n-1}}{\text{Ker}(f)}\cong \im(f)\to [\h^{i+1}_{\U}(F_\m(\M))]_n\to  \im(g)\to 0.  
    $$
     which is induced by \eqref{8.2-induced}.
     By induction hypothesis $\ell_A\Big([\h^i_{\mathcal{U}}(N)]_{n-1}\Big)< \infty$ and by Theorem \ref{m1}, $\ell_A\Big([\h^{i+1}_{\mathcal{U}}(G(\M_L))]_n\Big)<\infty.$ Therefore $\ell_A\Big([\h^{i+1}_{\U}(F_\m(\M))]_n\Big)<\infty$ for all $n.$ Now we want to prove $\ell_A\Big([\h^{i+1}_{\mathcal{U}}(N)]_{n}\Big)<\infty$ for all $n.$  Consider the exact sequence 
    $$
    0 \to \frac{[ \h^{i}_{\mathcal{U}}(F_\m(\M))]_{n}}{\text{Ker}(p)}\cong \im(p)\to [\h^{i+1}_{\U}(N)]_n\to  \im(q)\to 0.  
    $$
      induced by \eqref{8.3-induced}. 
    Since by induction hypothesis $\ell_A\Big([\h^i_{\mathcal{U}}(F_\m(\M))]_{n}\Big)< \infty$ and by Theorem \ref{m1},  $\ell_A\Big([\h^{i+1}_{\mathcal{U}}(G(\M))]_n\Big)<\infty.$ Therefore $\ell_A\Big([\h^{i+1}_{\U}(N)]_n\Big)<\infty$ for all $n.$
\end{proof}


\begin{theorem}\label{theorem-fiber-1}
	Suppose the same hypothesis as in Lemma \ref{fiber-lemma-1} holds true. Then 
	$F_\m(\M)$ is generalized Cohen-Macaulay $R(I)$-module. Moreover,  if $\depth M>0$, then  $\depth F_\m(\M)=\depth M.$
	\end{theorem}

\begin{proof}
By Lemma \ref{fiber-lemma-2} and Lemma \ref{fiber-lemma-3}, $F_\m(\M)$ is generalized Cohen-Macaulay $R(I)$-module. Now we prove $\depth F_\m(\M)=\depth M$, provided $\depth M>0.$ Let $\depth M= t$ where $1\leq t \leq d.$ Hence $\h^i_\m(M)=0$ for $0\leq i \leq t-1$ and $\h^t_\m(M)\neq 0.$ First we show that $\h^i_\U(F_\m(\M))=0$ for $0\leq i \leq t-1.$ We prove by induction on $t.$ If $t=1$, then it follows from Theorem \ref{fiber-lemma-1}. Suppose $t\geq 2$ and result holds for $1\leq s< t.$ By induction hypothesis, we have $\h^i_\U(F_\m(\M))=0$ for $0\leq i \leq s-1.$ Consider the exact sequence induced by the short exact sequence \eqref{1st-exact-eqn},
$$
\h^{s-2}_\U(F_\m(\M))
\to \h^{s-1}_\U(N)\to \h^{s-1}_\U(G(\M)). $$
By induction hypothesis, $\h^{s-2}_\U(F_\m(\M))=0$ and by Theorem \ref{m1}, we have $\h^{s-1}_\U(G(\M))=0$. This gives $\h^{s-1}_\U(N)=0.$ Now we consider the exact sequence induced by the short exact sequence \eqref{2nd-exact-eqn},
$$
\h^{s-1}_\U(N(-1))\to \h^s_\U(F_\m(M))\to \h^s_\U(G(\M_L)). 
$$
Since $\h^{s-1}_\U(N(-1))=0$ and by Theorem \ref{m1}, $\h^s_\U(G(\M_L))=0,$ we get $\h^s_\U (F_\m(\M))=0.$ So we conclude that $\h^i_\U(F_\m(\M))=0$  $\forall ~ 0\leq i \leq t-1.$ Now we prove $\h^t_\U(F_\m(\M))\neq 0$. 

Since $\h^{t-1}_\U(F_\m(\M))=0$ and by Theorem \ref{proposition-5.3}, $[\h^t_\U(G(\M))]_{1-t}=0$, we get $[\h^t_\U(N)]_{1-t}=0$ using the sequence \eqref{long-exact-seq-1}. This implies $[\h^t_\U(N(-1))]_{2-t}=0.$ Further, since $[\h^t_\U(G(\M_L))]_{2-t}\neq 0$, we get $[\h^t_\U(F_\m(\M))]_{2-t}\neq 0$ looking at the exact sequence \eqref{long-exact-seq-2}. This gives  $\depth F_\m(\M)= \depth M .$
\end{proof}
Finally to obtain Theorem \ref{theorem-fiber-m1} as stated in Introduction of the paper, we consider the case with $M=A$ and $\M=\{I^n\}_{n\geq 0}$, we write 
$f_0(I)$, in place of $f_0^I(M)$ which is the multiplicity of $F_\m(I)=\mathop\oplus\limits_{n\geq 0}\frac{I^n}{\m I^n}$. Then $$f_0(I)= e_1(I)-e_0(I)-e_1(J)+\ell_A(A/\m)+\ell_A(A/(\m I+J))$$ if and only if $f_0(I)=e_1(I)-e_1(J)-e_0(I)+\ell_A(A/I)-\mu(I)-d+1$. We get the following result as a corollary of Theorem \ref{theorem-fiber-1}. 
\begin{corollary}\label{corollary_8.5}
	Let $(A,\m)$ be a Noetherian local ring,  $J=(a_1,\ldots,a_d)\subseteq I$  a reduction of $I$. Let $a_1,\ldots,a_d$  satisfy $(C_1)$ and $(C_2)$ for $A$.  Suppose 	\begin{align}\label{8.7}
		e_1(I)-e_1(J)&=2e_0(I)-2\ell_A(A/I)-\ell_A(I/(I^2+J))
        \text{ and }\\
		\label{8.8}f_0(I)&= e_1(I)-e_0(I)-e_1(J)+\ell_A(A/I)-\mu(I)-d+1.	
	\end{align} Then  
	$F_\m(I)$ is generalized Cohen-Macaulay with $\depth F_\m(I)=\depth A$, provided $\depth A>0.$
	\end{corollary}
In order to find the geometric regularity, we need the following lemma. 

\begin{lemma} Suppose the same hypothesis as in Lemma \ref{fiber-lemma-1} holds. Then $a_i(F_\m(M))=\left\lfloor \frac{i}{2} \right\rfloor+1$ and $a_i(N)=\left\lfloor \frac{i-1}{2} \right\rfloor+1$ for all $2\leq i\leq d$ where $\left\lfloor i \right\rfloor=\{x\in \mathbb{Z}: x\leq i < x+1\}$ and $a_i(*):=\max\{j~|~H^i_{-}(*)_j\neq 0\}$. 
\end{lemma}
\begin{proof}
   We proceed by induction on $i$.
   Let $i=2$. 
    By Theorem \ref{proposition-5.3}, $\h^1_\U(G(\M_L))=[\h^1_{\mathcal{U}}(G(\M_L))]_1$ and $\h^2_\U(G(\M_L))=[\h^2_{\mathcal{U}}(G(\M_L))]_0$  while by Lemma \ref{fiber-lemma-1},  $\h^1_{\mathcal{U}}(N)=[\h^1_{\mathcal{U}}(N)]_{1}$. Using the long exact sequence \eqref{long-exact-seq-2}, we obtain $a_2(F_\m(\M))=2.$ To compute $a_2(N)$, again
     by Theorem \ref{proposition-5.3}, $\h^1_\U(G(\M))=[\h^1_{\mathcal{U}}(G(\M))]_1$ and $\h^2_\U(G(\M))=[\h^2_\U(G(\M))]_0$ while by Lemma \ref{fiber-lemma-1}, we have $\h^1_\U(F_\m(\M))=[\h^1_\U(F_\m(\M))]_1.$ From the long exact sequence \eqref{long-exact-seq-1}, we get $a_2(N)=1.$
    Now let $3\leq i\leq d-1$ and the assertion be true for $i$.

  For $a_{i+1}(F_\m(\M))$, 
 by Theorem \ref{proposition-5.3},  $\h^{i}_{\mathcal{U}}(G(\M_L))=[\h^{i}_{\mathcal{U}}(G(\M_L))]_{2-i}$, $\h^{i+1}_{\mathcal{U}}(G(\M_L))=[\h^{i+1}_{\mathcal{U}}(G(\M_L))]_{1-i}$ and by the induction hypothesis, $a_i(N)=\left\lfloor \frac{i-1}{2} \right\rfloor+1.$ Hence, in sequence \eqref{long-exact-seq-2} highest non-vanishing graded component of $\h^{i+1}_{\U}(F_\m(\M))$ occurs in degree $\left\lfloor \frac{i-1}{2} \right\rfloor+2$, so that $a_{i+1}(F_\m(\M))=\left\lfloor \frac{i+1}{2} \right\rfloor+1.$
    Finally, to compute $a_{i+1}(N)$, again 
 by Theorem \ref{proposition-5.3}, $\h^{i}_{\mathcal{U}}(G(\M))=[\h^{i}_{\mathcal{U}}(G(\M))]_{2-i}$, $\h^{i+1}_{\mathcal{U}}(G(\M))=[\h^{i+1}_{\mathcal{U}}(G(\M))]_{1-i}$ and by induction hypothesis $a_i(F_\m(\M))=\left\lfloor \frac{i}{2} \right\rfloor+1.$ Hence in the sequence \eqref{long-exact-seq-1} the highest non-vanishing graded piece of $\h^{i+1}_\U(N)$ occurs in degree $\left\lfloor \frac{i}{2} \right\rfloor+1$, so that $a_{i+1}(N)=\left\lfloor \frac{i}{2} \right\rfloor+1.$  
\end{proof}
For any finitely generated graded $A$-module $M$ geometric regularity of $M$ as
\begin{equation}\label{Def_g-reg}
\g-reg(M):=\max\{a_i(M)+i~|~i\geq 1\}.    \end{equation}

\begin{corollary}
   With the same hypothesis as in Lemma \ref{fiber-lemma-1}, we have $\g-reg(F_\m(\M))=\left\lfloor \frac{d}{2} \right\rfloor+1+d$. 
\end{corollary}
\begin{corollary}\label{corollay_8.8}
    With the same hypothesis as in Corollary \ref{corollary_8.5}, we have $\g-reg(F_\m(I))=\left\lfloor \frac{d}{2} \right\rfloor+1+d.$ 
\end{corollary}
In the next theorem, we prove the passage of Buchsbaumness to $F_\m(\M)$ for a Buchsbaum module of almost maximal depth. In dimension two, one can remove the depth condition and conclude that 	$F_{\m}(\M)/ \h^0_\U(F_\m (\M))$ is Buchsbaum.
\begin{theorem}\label{theorem-fiber-2}  Let $M$ be a faithful flat Buchsbaum $A$-module with $\depth M\geq d-1.$ Suppose $(C_2)$ condition is satisfied by $a_1,\ldots,a_d$ for both $A$ and $M.$ Suppose  the equalities  in \eqref{the-theorem-2-condition1} and \eqref{the-theorem-2-condition2} hold. Then 
	$F_\m(\M)$ is Buchsbaum $R(I)$-module.
    \end{theorem}
\begin{proof}
	 Let $d=1$. 	Since $\m\cdot \h^0_\m(M)=0$, we have    $\mathcal{U}\cdot\h^0_{\mathcal{U}}(F_\m(\M))=0$ which implies  $F_\m(\M)$  is Buchsbaum $R(I)$-module. Suppose $d\geq 2$ and the assertion holds for $d-1$. By Theorem \ref{m1} and Theorem \ref{theorem-fiber-1}, $G(\M)$, $G(\M_L )$ and
	 $F_{\m }(\M)$ have positive depths. Let $M^\prime=M/a_1M$ and $f=a_1t\in R(I).$  Then $M^\prime$ is Buchsbaum, the condition $(C_2)$ is satisfied by images of $a_1,\ldots,a_d$ in $A^\prime$ for both $A^\prime$ and $M^\prime.$ Moreover, the equalities in \eqref{the-theorem-2-condition1} and \eqref{the-theorem-2-condition2} are satisfied for $\M^\prime$. By induction hypothesis, $F_{\m A^\prime}(\M^\prime)$ is Buchsbaum. 
	 Since $f=a_1t$ is a regular element on both $F_{\m}(\M)$ and $G(\M )$, we have $F_{\m A^\prime}(\M^\prime)\cong F_{\m }(\M)/fF_{\m}(\M)$.  By \cite[I. Proposition 2.19]{SVbook}, $F_{\m}(\M)$ is Buchsbaum module. 
	\end{proof}
    \begin{corollary}\label{theorem-fiber-2(a)}
        Let $A$ be a Buchsbaum local ring with $\depth A\geq d-1$ and   $J=(a_1,\ldots,a_d)\subseteq I$ a reduction of $I$  such that $(a_1,\ldots, \check{a_i},\ldots,a_d):a_i\subseteq I$ for $1\leq i\leq d$. Suppose the equalities in \eqref{8.7} and \eqref{8.8} hold. Then 
		$F_\m(I)$ is Buchsbaum $R(I)$-module. 
    \end{corollary}
\begin{corollary}Let $M$ be a two-dimensional faithful flat Buchsbaum $A$-module. Let the rest of the hypothesis be the same as in 
	 Theorem \ref{theorem-fiber-2}. Then 
	$F_{\m}(\M)/ \h^0_\U(F_\m (\M))$ is Buchsbaum.
	\end{corollary}
	 \begin{proof}
	 	 Let $\overline{M}=M/W$. Then $\depth \overline{M}\geq 1$, $\overline{M}$ is Buchsbaum $\overline{A}$-module,
	 	the conditions  $(C_2)$ and the equalities in \eqref{the-theorem-2-condition1} and \eqref{the-theorem-2-condition2} are satisfied for $\overline{M}$. Therefore, $F_{\m \overline{A}}(\overline{\M})$ is Buchsbaum. Further, the kernel of $F_{\m}(\M)\to F_{\m \overline{A}}(\overline{\M})$ is $\h^0_\U(F_\m (\M))$, i.e., $F_{\m \overline{A}}(\overline{\M})\cong F_{\m}(\M)/ \h^0_\U(F_\m (\M)).$ 
	 \end{proof}
     \begin{corollary}\label{corollay_8.12}
         Suppose $A$ is a Buchsbaum local ring of dimension two and $\depth A=0$.
	Let the rest of the hypothesis be the same as in 
	Corollary \ref{theorem-fiber-2(a)}. Then 
	$F_{\m}(I)/ \h^0_\U(F_\m (I))$ is Buchsbaum.
     \end{corollary}
\section*{Acknowledgements}

Both the authors are partially supported by the SERB (ANRF) Grant No.~SPG/2022/002099, Government of India. 
AKY was partially supported by the UGC Fellowship (NTA Ref.\ No.~191620042352), Government of India, during the initial phase of this paper. 
KS was supported by the Postdoctoral Fellowship of Chennai Mathematical Institute at the time this work was initiated.
     

\end{document}